\documentclass{article}
\usepackage{amssymb,amsmath,amsfonts,amsthm}
\usepackage{color}
\usepackage{authblk}
\usepackage[dvipsnames]{xcolor}
\usepackage{graphicx}
\usepackage{placeins}

\newtheorem{theorem}{Theorem}[section]
\newtheorem{lemma}[theorem]{Lemma}
\newtheorem{proposition}[theorem]{Proposition}
\newtheorem{corollary}[theorem]{Corollary}

\newtheorem*{example}{Example}
\newtheorem*{examples}{Examples}
\newtheorem*{remark}{Remark}

\numberwithin{equation}{section}

\newcommand{\abs}[1]{\lvert{#1}\rvert}    

\newcommand{\cX}{{\mathcal X}} 

\renewcommand{\epsilon}{\varepsilon}
\newcommand{\Ric}{\mathrm{Ric}} 
\DeclareMathOperator{\grad}{grad} 

\newcommand{\RR}{\mathbb{R}}
\newcommand*{\CC}{\mathbb{C}}
\newcommand*{\HH}{\mathbb{H}}
\newcommand*{\NN}{\mathbb{N}}
\newcommand*{\ZZ}{\mathbb{Z}}
\newcommand*{\CP}{{\mathbb{C}}P}

\title{Eigenvalue estimates for the magnetic Hodge Laplacian on differential forms}

\author[1]{Michela Egidi\thanks{\texttt{michela.egidi@uni-rostock.de}}}
\author[2]{Katie Gittins\thanks{\texttt{katie.gittins@durham.ac.uk}}}
\author[3,4]{Georges Habib\thanks{\texttt{ghabib@ul.edu.lb}}}
\author[2]{Norbert Peyerimhoff\thanks{\texttt{norbert.peyerimhoff@durham.ac.uk}}}

\affil[1]{\footnotesize Universit\"at Rostock, Institut f\"ur Mathematik, 18051 Rostock, Germany}
\affil[2]{\footnotesize Department of Mathematical Sciences, Durham University, Mathematical Sciences and Computer Science Building, Upper Mountjoy Campus, Stockton Road, Durham University, DH1 3LE, United Kingdom}
\affil[3]{\footnotesize Lebanese University, Faculty of Sciences II, Department of Mathematics, P.O. Box 90656 Fanar-Matn, Lebanon}
\affil[4]{\footnotesize Universit\'e de Lorraine, CNRS, IECL, 54506 Nancy, France}

\date{}

\begin{document}

\maketitle

\vspace*{-1cm}

\begin{abstract}
     In this paper we introduce the magnetic Hodge Laplacian, which is a generalization of the magnetic Laplacian on functions to differential forms. We consider various spectral results, which are known for the magnetic Laplacian on functions or for the Hodge Laplacian on differential forms, and discuss similarities and differences of this new ``magnetic-type'' operator.
\end{abstract}

\tableofcontents

\section{Introduction and statement of results}

The classical magnetic Laplacian on a Riemannian manifold $(M^n,g)$ associated to a smooth real $1$-form $\alpha \in \Omega^1(M)$ acts
on the space of smooth complex-valued functions $C^\infty(M,\CC)$ and is given by
\begin{equation} \label{eq:Dalpha}
\Delta^\alpha  = \delta^\alpha d^\alpha ,
\end{equation}
where $d^\alpha := d^M + i \alpha $ and $\delta^\alpha:=\delta^M-i\langle\alpha^\sharp,\cdot\rangle$ (note that $\delta^M$ is the $L^2$-adjoint of $d^M$). Here $\alpha^\sharp \in \cX(M)$ is the vector field corresponding to the $1$-form $\alpha$ via the musical isomorphism
$\langle \alpha^\sharp,X \rangle = \alpha(X)$. The $1$-form $\alpha$ is called the \emph{magnetic potential} and $d^M \alpha$ is  the \emph{magnetic field}. The magnetic Laplacian $\Delta^\alpha$ can be viewed as a first order perturbation of the usual Laplacian $\Delta^M = \delta^M d^M$, namely for any $f\in C^\infty(M,\CC)$,
\begin{equation}\label{eq:DalphaD}
\Delta^\alpha f = \Delta^M f - 2 i \langle \grad f, \alpha^\sharp \rangle + ( |\alpha^\sharp|^2 - i\, {\rm div}\, \alpha^\sharp) f.
\end{equation}

In the case of a closed manifold or a compact manifold with boundary, both operators $\Delta^M$ and $\Delta^\alpha$ (with suitable boundary conditions when $\partial M \neq \emptyset$) have a discrete spectrum with non-decreasing
eigenvalues with multiplicity denoted by $(\lambda_k(M))_{k \in \NN}$ and $(\lambda_k^\alpha(M))_{k \in \NN}$, respectively. There are very few Riemannian manifolds where the complete set of eigenvalues can be given  explicitly. Amongst them is the unit round sphere $\mathbb{S}^n$ with the standard metric $g$, whose eigenfunctions can be described as spherical harmonics. In Appendix \ref{sec:berger}, we give an explicit derivation of the spectrum of a magnetic Laplacian on $(\mathbb{S}^3,g)$ with a special magnetic potential $\alpha$. This derivation is based on the Hopf fibration $\mathbb{S}^1 \hookrightarrow \mathbb{S}^3 \rightarrow \mathbb{S}^2$, and $\alpha$ is a constant magnetic field along the $\mathbb{S}^1$-fibers.

In analogy with the generalization of the usual Laplacian $\Delta^M$ on functions to the Hodge Laplacian $\delta^M d^M + d^M \delta^M$ on differential forms,  it is natural to generalize the magnetic Laplacian on functions to complex differential forms as follows. On the set of complex-valued differential $p$-forms $\Omega^p(M,\CC)$, we define
$$ \Delta^\alpha:= \delta^\alpha d^\alpha + d^\alpha \delta^\alpha $$
where $d^\alpha:= d^M+ i \alpha \wedge$ and  $\delta^\alpha:=\delta^M-i\alpha^\sharp\lrcorner$ is its formal adjoint. Both $d^\alpha$ and $\delta^\alpha$ can also be expressed via the {\it magnetic covariant derivative}
$\nabla^\alpha_X Y:=\nabla^M_X Y+i\alpha(X) Y$ for any $X,Y\in C^\infty(TM\otimes\mathbb{C})$ (see formula \eqref{eq:localddelta}).
We refer to this operator $\Delta^\alpha$ acting on $\Omega^p(M,\CC)$ as the \emph{magnetic Hodge Laplacian} on complex $p$-forms.

\medskip

We establish the following results for the magnetic Hodge Laplacian on an oriented Riemannian manifold $(M^n,g)$:
\begin{itemize}
    \item[(a)] We show that the magnetic Hodge Laplacian commutes with the Hodge star operator (see Corollary \ref{cor:maglapstar}).
    \item[(b)] We derive a magnetic analogue of the classical Bochner-Weitzenb\"ock formula  (see Theorem \ref{thm:magboch}).
    \item[(c)] We prove gauge invariance of the magnetic Laplacian on forms $\Delta^\alpha$ (see Corollary \ref{cor:gaugeinv}).
    \item[(d)] We obtain a Shigekawa-type result (see Theorem \ref{thm:shiforms}) for the magnetic Hodge Laplacian
    $\Delta^{\alpha}$ on a closed Riemannian manifold $M$ in the case where $M$ has a parallel $p$-form and $\alpha$ is a Killing $1$-form (for the original statement, see \cite{Sh87}).
    \item[(e)] Following a result by Gallot-Meyer \cite{GM:75} for the Hodge Laplacian, we derive a lower bound for the first eigenvalue of the magnetic Hodge Laplacian for closed manifolds (see Theorem \ref{thm:gm}).
    \item[(f)] Following a result by Colbois-El Soufi-Ilias-Savo \cite{CESIS-17} for the magnetic Laplacian on functions, we derive an upper bound for the first eigenvalue of the magnetic Hodge Laplacian for closed manifolds (see Theorem \ref{thm:CS}).
    \item[(g)] We show that in general the diamagnetic inequality does not hold for magnetic Hodge Laplacians (Corollary \ref{cor:notdiamagineq}). In fact, we give a counterexample which is based on the calculations in Appendix \ref{sec:berger}. In addition, we give an explicit characterization which determines when the diamagnetic inequality holds for $\Delta^{t \xi}$ with $\xi$ a Killing vector field (see Corollary \ref{cor:diam}).
    \item[(h)] Following the work of Raulot-Savo in \cite{RS:11}, we derive a Reilly formula for the magnetic Hodge Laplacian on Riemannian manifolds with boundary (see Theorem \ref{thm:reilly}) and use it to derive a lower bound for the first eigenvalue of the magnetic Hodge Laplacian on an embedded hypersurface of a Riemannian manifold (see Theorem \ref{thm:rsestimate}).
    \item[(i)] Following the work of Guerini-Savo in \cite{GS}, we derive a ``gap'' estimate between the first eigenvalues of consecutive $p$-values of the magnetic Hodge Laplacians on $\Omega^p(M,\CC)$ for isometrically immersed manifolds $(M^n,g)$ in Euclidean space $\RR^{n+m}$ (see Theorem \ref{gapestimatethm}).
\end{itemize}

\noindent{\bf Acknowledgment:} The third named author thanks Durham University for its hospitality during his stay. He also thanks the Alexander von Humboldt foundation and the Alfried Krupp Wissenschaftskolleg in Greifswald.

\section{Review of the magnetic Laplacian for functions}
\label{sec:maglapfunc}

Before we introduce the magnetic Hodge Laplacian in the next section, we first recall some results for the classical magnetic Laplacian on functions. Let $(M^n,g)$ be a Riemannian manifold  and $\alpha \in \Omega^1(M)$. The magnetic Laplacian $\Delta^\alpha$ acting on complex-valued smooth functions defined by formula \eqref{eq:Dalpha} has the property of gauge invariance, that is $\Delta^\alpha(e^{if})=e^{if}\Delta^{\alpha+d^M f}$ for any smooth real-valued function $f$. When $M$ is compact  (with or without boundary), the spectrum of $\Delta^\alpha$ (or with suitable boundary conditions when $\partial M\neq\emptyset$) is discrete. Therefore, by the gauge invariance, the spectrum of $\Delta^\alpha$ is equal to the spectrum of $\Delta^{\alpha+d^M f}$. Thus, when $\alpha$ is exact, the spectrum of $\Delta^\alpha$ reduces to that of the usual Laplace-Beltrami operator. In \cite[Prop. 3]{CS18}, it is proven  that one can always assume that $\alpha$ is a co-closed $1$-form (and tangential, i.e. $\nu\lrcorner\alpha=0$, when $M$ has a boundary) without changing the spectrum of $\Delta^\alpha$. Moreover, by using the Hodge decomposition on compact manifolds, the authors show in \cite[Prop. 1]{CESIS-17}  that one can further consider $\alpha$ to be of the form
$$ \alpha = \delta^M \psi + h,$$
where $\psi$ is a $2$-form on $M$ (with $\nu\lrcorner\psi=0$ when $\partial M\neq\emptyset$), and $h$ is a harmonic $1$-form on $M$, that is, $d^M h = \delta^M h = 0$ (with $\nu\lrcorner h=0$ when $\partial M\neq\emptyset$), and again the spectrum does not change. Here, we point out that the first eigenvalue $\lambda^\alpha_{1}(M)$ of $\Delta^\alpha$ is not necessarily zero like for the usual Laplacian $\Delta^M$ as shown in \cite[Ex. 1]{Sh87}. This interesting property of the magnetic Laplacian was  characterized  by Shigekawa (see \cite[Prop. 3.1 and Thm. 4.2]{Sh87}) as follows.

\begin{theorem}[Shigekawa] \label{thm:shi}
 Let $(M^n,g)$ be a closed Riemannian manifold and
 $$ \mathfrak{B}_M = \left\{ \alpha_\tau := \frac{d^M\tau}{i\tau}: \tau\in C^\infty(M,\mathbb{S}^1) \right\}. $$
 Then the following are equivalent:
 \begin{itemize}
     \item[(a)] $\alpha \in \mathfrak{B}_M$,
     \item[(b)] $d^M \alpha = 0$ and $\int_c \alpha \in 2\pi \mathbb{Z}$ for all closed curves $c$ in $M$,
     \item[(c)] $\lambda_1^\alpha(M) = 0$.
 \end{itemize}
\end{theorem}

Hence, when $\alpha$ cannot be gauged away, meaning that $\alpha$ does not belong to the set $\mathfrak{B}_M$, the first eigenvalue is necessarily positive. This gauge invariance can be described by the following: If $\alpha_\tau \in \mathfrak{B}_M$ for some $\tau\in C^\infty(M,\mathbb{S}^1)$, the Laplacians $\Delta^\alpha$ and $\Delta^{\alpha + \alpha_\tau}$ are unitarily equivalent, that is
 $$ \bar \tau \Delta^\alpha \tau = \Delta^{\alpha+\alpha_\tau}.$$
Thus  $\Delta^\alpha$ and $\Delta^{\alpha+\alpha_\tau}$ have the same spectrum as stated before. Now, the \emph{diamagnetic inequality} compares the first eigenvalue of $\Delta^\alpha$ to the one for the Laplacian $\Delta^M$ and says that
$$\lambda_1^\alpha(M)\geq \lambda_1(M),$$
with equality if and only if the magnetic potential $\alpha$ can be gauged away. When $M$ has no boundary, the diamagnetic inequality provides no information since $\lambda_1(M)=0$. However, when we consider manifolds with boundary and the magnetic Laplacian is associated to the Dirichlet or Robin boundary conditions, the diamagnetic inequality still holds and tells us that the first eigenvalue $\lambda_1^\alpha(M)$ is always positive.

\medskip

A simple estimate for the first eigenvalue of the magnetic Laplacian can be deduced straightforwardly from the min-max principle. Indeed,
when applying the Rayleigh quotient to a constant function, we get, after choosing $\delta^M\alpha=0$, that
$$\lambda^\alpha_1(M)\leq \frac{\int_M|\alpha|^2 d\mu_g}{{\rm Vol}(M)}\leq ||\alpha||^2_\infty.$$

Several papers have been devoted to estimating the first eigenvalue of the magnetic Laplacian, see, for example, \cite{BBC:03, BDP:16, L:96, HOOO:99, LLPP:15, LS:15, CS18, CS21, CS:21, CS:22, ELMP:16, CESIS-17}.
Among these results, we quote two of them \cite{ELMP:16}, \cite{CESIS-17} on closed Riemannian manifolds. 

The first result gives magnetic  Lichnerowicz-type estimates for the first two eigenvalues:

\begin{theorem}[see {\cite[Thm. 1.1]{ELMP:16}}]\label{thm:Lichnerowicz}
	Let $(M^n,g)$ be a closed Riemannian manifold of dimension $n\geq 2$ and $\alpha \in \Omega^1(M)$. If
	\begin{equation} \label{eq:Lichcond}
	\Ric^M \geq C>0\qquad \text{and}\qquad\lVert d^M\alpha \rVert_\infty\leq \left(1+2\sqrt{\frac{n-1}{n}}\right)^{-1}C,
	\end{equation}
	then we have
	\begin{equation}\label{eq:Lich}
		0\leq\lambda^\alpha_1(M)\leq a_-(C,\lVert d^M\alpha\rVert_\infty,n)\qquad\text{ and }\qquad\lambda^\alpha_2(M)\geq a_+(C,\lVert d^M\alpha\rVert_\infty,n),
	\end{equation}
	where
	\begin{equation*}
		a_{\pm}(C,A,n)=n\cdot
		\frac{ (C-A)\pm\sqrt{(C-A)^2-4(\frac{n-1}{n})A^2 }}{2(n-1)}.
	\end{equation*}
\end{theorem}

The technique used to obtain this result is an integral Bochner-type formula which involves the magnetic Hessian that is associated to the magnetic covariant derivative $\nabla^\alpha$.  A related result to Theorem \ref{thm:Lichnerowicz} for the magnetic Laplacian with Robin boundary conditions on compact Riemannian manifolds $(M,g)$ with  smooth boundary was proved in \cite{HK:18}. In the setup of the above theorem, it is natural to ask whether the estimates are sharp for some $\alpha$ that is not gauged away. For this, we employ 
the example of the round sphere $\mathbb{S}^3$ where the magnetic field $\alpha$ is collinear to the Killing vector field that defines the Hopf fibration. We refer to Appendix \ref{sec:berger} for more details on the computation.

\begin{example}[Unit sphere $\mathbb{S}^3$ with $\alpha = t Y_2$]
  Let $(\mathbb{S}^3,g)$ be the unit sphere in $\mathbb{R}^4$ with standard metric $g$ of curvature $1$. We use the notation introduced in Appendix \ref{sec:berger}. Let $\alpha = t Y_2$ where $Y_2$ is the unit Killing vector field on $\mathbb{S}^3$.  Using \eqref{eq:exteriory2}, we obtain $d^M \alpha=2t\, Y_3 \wedge Y_4$ where $\{Y_2,Y_3,Y_4\}$ is an orthonormal frame of $T\mathbb{S}^3$
  and, therefore, $\Vert d^M \alpha \Vert_\infty = 2 t$. Since $\Ric^M=C=2$, condition \eqref{eq:Lichcond} is satisfied for $|t| \le \frac{\sqrt{3}}{\sqrt{3}+\sqrt{8}} = t_{\rm{max}} \approx 0.38$, and for $t \in [0,t_{\rm{max}}]$ we have, by \eqref{eq:Lich},
  \begin{multline*}
  \lambda_1^{\alpha}(\mathbb{S}^3) \le \frac{3}{2}\left[ (1-t) - \sqrt{(1-t)^2- \frac{8}{3}t^2} \right] \\ \le
  \frac{3}{2}\left[ (1-t) + \sqrt{(1-t)^2- \frac{8}{3}t^2} \right] \le \lambda_2^{\alpha}(\mathbb{S}^3). \end{multline*}
  On the other hand, we conclude from \eqref{eq:specmagS3} that $\lambda_1^{\alpha}(\mathbb{S}^3) = t^2$ and $\lambda_2^{ \alpha}(\mathbb{S}^3) = 3-2t+t^2$ for small $t \in [0,t_{\rm{max}}]$. The relations between these two smallest eigenvalues and their estimates for small $t > 0$ are illustrated in Figure \ref{fig:lichest}.

  \begin{figure}[h]
  \begin{center}
  \includegraphics[width=0.6 \textwidth]{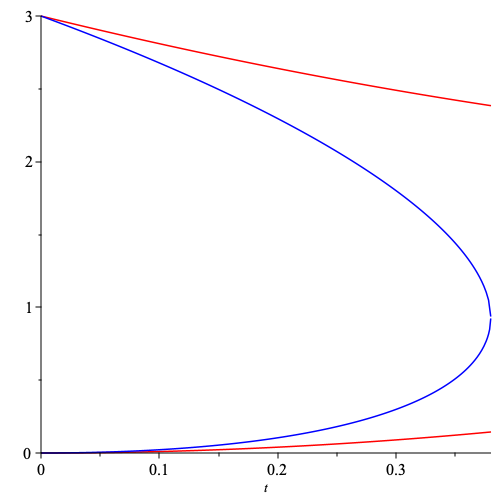}
  \caption{Eigenvalues $\lambda_1^{\alpha}(\mathbb{S}^3)$ and $\lambda_2^{\alpha}(\mathbb{S}^3)$ in red and upper and lower bounds in blue, as functions over $t \in [0,t_{\rm{max}}]$.} \label{fig:lichest}
  \end{center}
  \end{figure}
\end{example}

\FloatBarrier

\medskip

As we can see from Figure \ref{fig:lichest}, sharpness of the upper estimate of $\lambda_1^\alpha(\mathbb{S}^3)$ is lost 
(see the discussion after Lemma \ref{lem:gallotmeyer}).

The second result was given in \cite{CESIS-17} in the general setting of magnetic Schr\"odin\-ger operators $\Delta^\alpha + q$ with Neumann boundary conditions. For simplicity, we formulate it in the special case of a closed Riemannian manifold $(M^n,g)$ with vanishing potential $q=0$. We will return to this estimate later in Subsection \ref{subsec:CESIS}.

\begin{theorem}[{\cite[Thm. 2]{CESIS-17}}] \label{thm:cesis17}
  Let $(M^n,g)$ be a closed Riemannian manifold and let $\alpha \in \Omega^1(M)$ be of the form $\alpha = \delta^M \psi + h$ with $\psi \in \Omega^2(M)$ and $h$ a harmonic $1$-form.  Then,
  $$ \lambda_1^\alpha(M) \le \frac{1}{{\rm{vol}}(M)}\left( d(h,\mathfrak{L}_\ZZ)^2+ \frac{\Vert d^M \alpha \Vert_2^2}{\lambda_{1,1}''(M)}\right), $$
  where $\lambda_{1,1}''(M)$ is the first eigenvalue of the Hodge Laplacian $\Delta^M$ on co-exact $1$-forms, $\mathfrak{L}_\ZZ$ is the lattice of integer harmonic $1$-forms in $\Omega^1(M)$, and
  $$ d(h,\mathfrak{L}_\ZZ)^2 = \inf_{\eta \in \mathfrak{L}_\ZZ} \Vert h - \eta \Vert_2^2. $$
\end{theorem}

In order to check the sharpness of this inequality, we consider again the case of the round sphere with the magnetic field given by the Killing vector field.
\begin{example}[Unit sphere $\mathbb{S}^3$ with $\alpha = t Y_2$] Let  $(\mathbb{S}^3,g)$ be the unit round sphere in $\RR^4$ with standard metric $g$ of curvature $1$ and let $\alpha = t Y_2$. Since $H^1(\mathbb{S}^3) = 0$ and $\delta^M \alpha = 0$, $\alpha$ is co-exact and therefore of the form $\delta^M \psi$ for some $\psi \in \Omega^2(M)$. Moreover, we have from \cite[p. 37]{GM:75}, \cite{Paq79}
  that
  $\lambda_{1,1}''(\mathbb{S}^3) = 4$. Thus, Theorem \ref{thm:cesis17} yields
  $$ \lambda_1^\alpha(\mathbb{S}^3) \le \frac{1}{4 \rm{vol}(\mathbb{S}^3)} \int_{\mathbb{S}^3} |d^M\alpha|^2 d\mu_g = t^2, $$
  that is, the upper estimate of the first magnetic eigenvalue is sharp for this case.
\end{example}

\medskip

Finally, as we mention in the introduction, examples of closed Riemannian manifolds $(M^n,g)$ with non-trivial magnetic potential $\alpha \in \Omega^1(M)$ (that is, magnetic potential which cannot be gauged away) for which the full spectrum of the magnetic Laplacian $\Delta^\alpha$ can be explicitly given, are very scarce (see, for example, \cite{CS18, CS:22} for such computations).

\section{The magnetic Hodge Laplacian for differential forms}

In this section, we introduce the magnetic Hodge Laplacian for differential forms, prove a magnetic Bochner formula, and discuss its gauge invariance. Henceforth $(M^n,g)$ will denote an oriented $n$-dimensional Riemannian manifold and $\Omega^p(M)$ and $\Omega^p(M,\CC)$ will denote the spaces of real and complex differential $p$-forms for $0 \le p \le n$. The spaces of real and complex vector fields on $M$ are denoted by $\cX(M)$ and $\cX_\CC(M)$.
To simplify notation, we will often identify real and complex vector fields with real and complex 1-forms via the (complex-linear) musical isomorphisms. That is, $\Omega^1(M,\mathbb{C})\to \cX_\CC(M);\, \omega\mapsto \omega^\sharp$ given by $\omega(X)=\langle X,\overline{\omega^\sharp}\rangle$, where $\langle \cdot,\cdot\rangle$ stands for the Hermitian scalar product extended from the Riemannian metric $g$ to $\cX_\CC(M)$.

\subsection{The magnetic Hodge Laplacian}

Fix a smooth $1$-form $\alpha \in \Omega^1(M)$ (a magnetic potential) and consider the \emph{magnetic differential} on $\Omega^p(M,\mathbb{C})$, given by
$$d^\alpha:=d^M+i\alpha\wedge.$$
It is not difficult to check that the $L^2$-adjoint of $d^\alpha$ acting on complex differential forms (when $M$ is without boundary) w.r.t. the Hermitian inner product
$$
\int_M \langle \omega, \eta \rangle \, d\mu_g = \int_M *(\omega \wedge * \bar \eta)
\, d\mu_g
$$
is given by
$$\delta^\alpha:=\delta^M-i\alpha^\sharp\lrcorner,$$
where $\delta^M = (-1)^{n(p+1)+1} * d^M *$ is the formal adjoint of $d^M$ on $p$-forms (both extended complex linearly to complex differential forms) and the Hodge star operator is extended to a complex linear operator  $*: \Omega^p(M,\mathbb{C}) \to \Omega^{n-p}(M,\mathbb{C})$. Recall here that the interior product ``$\lrcorner$" is the pointwise adjoint of the wedge product ``$\wedge$". Both $d^\alpha$ and $\delta^\alpha$ are the differential and co-differential associated to the magnetic connection on differential forms $\nabla^\alpha_X:=\nabla^M_X+i\alpha(X)$ on $\Omega^p(M,\mathbb{C})$. That means we
have
\begin{equation}\label{eq:localddelta}
d^\alpha=\sum_{j=1}^n e_j^*\wedge \nabla^\alpha_{e_j}\quad\text{and}\quad \delta^\alpha=-\sum_{j=1}^n e_j\lrcorner \nabla^\alpha_{e_j},
\end{equation}
where $\{e_1,\ldots,e_n\}$ is a local orthonormal frame of $TM$. Now, we define the \emph{magnetic Hodge Laplacian} acting on $\Omega^p(M,\mathbb{C})$ as follows:
$$
\Delta^\alpha:=d^\alpha\delta^\alpha+\delta^\alpha d^\alpha.
$$
We first have the following observation:

\begin{lemma} \label{lem:dalphadeltaalpha}
On differential $p$-forms, we have $* d^\alpha = (-1)^{p+1} \delta^\alpha*$ and $* \delta^\alpha = (-1)^p d^\alpha* $.
\end{lemma}

\begin{proof}
The proof is  straightforward from the fact that $* d^M = (-1)^{p+1} \delta^M*$ and $*(\alpha\wedge)=(-1)^p\alpha^\sharp\lrcorner*$ on $p$-forms. Also, we have that $* \delta^M = (-1)^p d^M* $ and $* (\alpha^\sharp \lrcorner ) = (-1)^{p+1} \alpha \wedge * $.
\end{proof}

The following is an immediate consequence of Lemma \ref{lem:dalphadeltaalpha} above.

\begin{corollary} \label{cor:maglapstar}
The magnetic Hodge Laplacian $\Delta^\alpha$ commutes with the Hodge star operator.
\end{corollary}

\begin{proof}
Indeed, on $p$-forms, we have
\begin{eqnarray*}
\Delta^\alpha*&=&(d^\alpha\delta^\alpha+\delta^\alpha d^\alpha) *\\
&=&(-1)^{p+1}d^\alpha*d^\alpha+(-1)^p\delta^\alpha*\delta^\alpha\\
&=&*(\delta^\alpha d^\alpha+d^\alpha\delta^\alpha)=*\Delta^\alpha.
\end{eqnarray*}
\end{proof}

The magnetic Laplacian $\Delta^\alpha$ has the same principal symbol as the Hodge Laplacian $\Delta^M$ (see Equation \eqref{eq:deltaalphaforms} in the next section), since it differs by lower order terms. Therefore, it is an elliptic, essentially self-adjoint operator acting on smooth complex forms on a closed oriented Riemannian manifold or acting on smooth complex forms with Dirichlet boundary condition on an oriented Riemannian manifold with boundary (see Subsection \ref{subsec:greensformula} below). Therefore, $\Delta^\alpha$ has a discrete spectrum consisting of nonnegative eigenvalues $( \lambda_{j,p}^\alpha(M) )_{j \in \NN}$, denoted in ascending order with multiplicities. Moreover, as for the usual Hodge Laplacian, its spectrum  on $p$-forms is the same as the one on $(n-p)$-forms and the first eigenvalue is characterized by
\begin{equation} \label{eq:minmax}
\lambda_{1,p}^\alpha(M) ={\rm inf} \left\{ \frac{\int_M (|d^\alpha \omega|^2 + |\delta^\alpha \omega|^2) d\mu_g}{\int_M |\omega|^2 d\mu_g} \right\},
\end{equation}
where $\omega$ runs over all smooth $p$-forms with $\omega\vert_{\partial M} = 0$, if $\partial M \neq \emptyset$.

We also note that the differential $d^\alpha$ does not satisfy the crucial property $d^\alpha \circ d^\alpha = 0$ to introduce cohomology groups. In fact, we have
\begin{equation}\label{eq:dalpha2}
     (d^\alpha)^2 = i d^M\alpha \wedge
\end{equation}
where $d^M\alpha \in \Omega^2(M)$ is the magnetic field. We
could, however, still define magnetic Betti numbers  
via
$$ b_j^\alpha(M) = \dim {\rm Ker}(\Delta^{\alpha}|_{\Omega^j(M,\mathbb{C})}). $$
Corollary \ref{cor:maglapstar} implies that $b_j^\alpha(M) = b_{n-j}^\alpha(M)$. 
Moreover, we have $b_0^\alpha(M) = b_n^\alpha(M) = 0$ for any magnetic potential $\alpha$ that cannot be gauged away, that is $\alpha\notin \mathfrak{B}_M$, by the diamagnetic inequality. In Theorem \ref{thm:shiforms}, we investigate the existence of closed Riemannian manifolds $(M^n,g)$ with a magnetic potential $\alpha$ that cannot be gauged away, for which some of the corresponding magnetic Betti numbers $b_k^\alpha(M)$, $1 \le k \le n-1$, are non-zero.

\subsection{A magnetic Bochner formula}

Recall that the Hodge Laplacian $\Delta^M:=d^M\delta^M+\delta^M d^M$ is related to
the Bochner Laplacian on $M$ via a curvature term  by the {\it Bochner-Weitzenb\"ock} formula. Namely, we have (see, e.g, \cite[Thm. 7.4.5]{Pet98} or \cite[p. 14]{Wu17})
\begin{equation} \label{eq:bochner}
\Delta^M=\nabla^*\nabla+\mathcal{B}^{[p]},
\end{equation}
where $\mathcal{B}^{[p]}$, called the \emph{Bochner operator}, is a symmetric endomorphism on $\Omega^p(M)$ given by $\mathcal{B}^{[p]}=\sum_{j,k=1}^n e_k^*\wedge e_j\lrcorner  R^M(e_j,e_k)$. Here $R^M$ is the curvature operator associated to the Levi-Civita connection $\nabla^M$ which is given by $R^M(X,Y)=[\nabla^M_X,\nabla^M_Y]-\nabla^M_{[X,Y]}$ for all $X,Y\in \cX(M)$ and $\{e_1,\ldots, e_n\}$ is a local orthonormal frame of $TM$. The Bochner Laplacian $\nabla^*\nabla$ is given by
$$\nabla^*\nabla=-\sum_{j=1}^n \nabla^M_{e_j}\nabla_{e_j}^M+\sum_{j=1}^n\nabla^M_{\nabla^M_{e_j}e_j}.$$

In the following, we derive a similar magnetic Bochner-Weitzenb\"ock formula for $\Delta^\alpha$, which will provide a relation between the Hodge Laplacians $\Delta^\alpha$ and $\Delta^M$. For this, we recall the following definition. Given a Euclidean vector space $V$ of dimension $n$ and an endomorphism $A: V \to V$, there exists a canonical extension $A^{[p]}$ of $A$ on the set of differential $p$-forms ($p\geq 1$) given by $A^{[p]}: \Lambda^p(V^*) \to \Lambda^p(V^*)$ via
\begin{equation}\label{eq:extension}
  (A^{[p]} \omega)(v_1,\dots,v_p) = \sum_{j=1}^p \omega(v_1,\dots,A v_j, \dots, v_p),
\end{equation}
for $v_1,\ldots,v_p \in V$. By convention, we take $A^{[0]}=0$. One can easily show from the definition that the endomorphism $A^{[p]}$ can be written in terms of $A$ as
 \begin{equation}\label{eq:extensionexpression}
 A^{[p]} = \sum_{j=1}^n e_j^*\wedge (A(e_j)\lrcorner),
 \end{equation}
where $\{e_1,\ldots,e_n\}$ is an orthonormal frame of $V$. If $A$ is a symmetric (resp. skew-symmetric) endomorphsim on $V$, then so is $A^{[p]}$ on $\Lambda^p(V^*)$. In this case, if we denote the eigenvalues of $A$ by $\eta_1\leq \ldots\leq \eta_n$, then we have the following estimates. For any $\omega\in \Lambda^p(V^*)$
\begin{equation}\label{eq:upperbounda}
\langle A^{[p]}\omega,\omega\rangle\geq \sigma_p |\omega|^2 \quad\text{and}\quad \langle A^{[p]}\omega,\omega\rangle\leq (\sigma_n-\sigma_{n-p}) |\omega|^2 \le p \Vert A \Vert \cdot |\omega |^2,
\end{equation}
where $\sigma_p:=\eta_1+\ldots+\eta_p$ are called the $p$-eigenvalues of $A^{[p]}$
and $\Vert A \Vert$ is the operator norm of $A$. In order to state the magnetic Bochner-Weitzenb\"ock formula, we introduce the following \emph{magnetic Bochner operator} on $\Omega^p(M,\CC)$:
$$ \mathcal{B}^{[p],\alpha}:=\sum_{j,k=1}^n e_k^*\wedge \left(e_j\lrcorner  R^\alpha(e_j,e_k)\right), $$
where as before $\{e_i\}_{i=1,\ldots,n}$ is a local orthonormal frame of $TM$. Here $R^\alpha$ is the curvature operator associated to the magnetic covariant derivative $\nabla^\alpha$, that is
$$R^\alpha(X,Y)Z = \nabla^\alpha_X \nabla^\alpha_Y Z - \nabla^\alpha_Y \nabla^\alpha_X Z - \nabla^\alpha_{[X,Y]} Z $$
for $X,Y,Z\in \cX_\mathbb{C}(M)$. Now, we express the magnetic Bochner operator in terms of the usual one by the following lemma.

\begin{lemma} \label{lem:magneticbochner} On the set of complex differential $p$-forms, the magnetic Bochner operator $\mathcal{B}^{[p],\alpha}$ is equal to
$$\mathcal{B}^{[p],\alpha}=\mathcal{B}^{[p]}-iA^{[p],\alpha},$$
where $A^{[p],\alpha}$ is the canonical extension to complex $p$-forms of the skew-symmetric endomorphism $A^\alpha$ on $TM$ given by $A^\alpha(X) = (X \lrcorner d^M\alpha)^\sharp$ for any vector field $X$ on $M$.
\end{lemma}

\begin{proof} An easy computation shows that, for any $X,Y\in \cX(M)$ and
$\omega \in \Omega^p(M,\CC),$
$$R^\alpha(X,Y)\omega=R^M(X,Y)\omega+i(d^M\alpha)(X,Y)\omega.$$
The proof can then be deduced from the definition of $\mathcal{B}^{[p],\alpha}$ and the fact that $A^{\alpha}$ is skew-symmetric.
\end{proof}

We make the following observation. Using the identity on $p$-forms $*(X^\flat\wedge)=(-1)^pX\lrcorner*$ valid for any vector field $X$, one can easily show that $\mathcal{B}^{[p]}=(-1)^{p(n-p)}*\mathcal{B}^{[n-p]}*$ which gives that $\langle\mathcal{B}^{[p]}\cdot,\cdot\rangle=\langle\mathcal{B}^{[n-p]}*\cdot,*\cdot\rangle$ where $*$ is the Hodge star operator on $M$ and $\langle\cdot,\cdot\rangle$ is the pointwise Hermitian product on $\Omega^p(M,\mathbb{C})$. In the same way, and since the endomorphism $A^{\alpha}$ is skew-symmetric, one can also show that $A^{[p],\alpha}=(-1)^{p(n-p)}*A^{[n-p],\alpha}*$. Therefore, we deduce that $\mathcal{B}^{[p],\alpha}=(-1)^{p(n-p)}*\mathcal{B}^{[n-p],\alpha}*$ and, thus,
\begin{equation}\label{eq:bochneroperatorstar}
\langle\mathcal{B}^{[p],\alpha}\cdot,\cdot\rangle=\langle\mathcal{B}^{[n-p],\alpha}*\cdot,*\cdot\rangle
\end{equation}
on complex $p$-forms. Notice here that $iA^{[p],\alpha}$ is a symmetric endomorphism on $\Omega^p(M,\mathbb{C})$. Now we formulate the magnetic Bochner-Weitzenb\"ock formula.

\begin{theorem} [Magnetic Bochner-Weitzenb\"ock formula] \label{thm:magboch}
Let $(M^n,g)$ be a Riemannian manifold and $\alpha \in \Omega^1(M)$. Then we have
\begin{equation}\label{eq:bochnermagnetic}
\Delta^\alpha=(\nabla^\alpha)^*\nabla^\alpha+\mathcal{B}^{[p],\alpha},
\end{equation}
where $(\nabla^\alpha)^*\nabla^\alpha=-\sum_{j=1}^n\nabla^\alpha_{e_j}\nabla^\alpha_{e_j}+\sum_{j=1}^n\nabla^\alpha_{\nabla_{e_j}^M e_j}$. Moreover, we have
\begin{equation} \label{eq:deltaalphaforms}
\Delta^\alpha = \Delta^M-i A^{[p],\alpha}+ i (\delta^M \alpha) - 2 i \nabla^M_\alpha+ |\alpha|^2.
\end{equation}
\end{theorem}

\begin{proof} The proof follows the same computations as for the Hodge Laplacian $\Delta^M$. For this, we use the expressions of $d^\alpha$ and $\delta^\alpha$ in \eqref{eq:localddelta} on an orthonormal frame $\{e_j\}_{j=1}^n$ on $TM$ chosen in a way that $\nabla^M e_j=0$ at some point $x\in M$. By the fact that, for all $X,Y\in \cX_{\mathbb{C}}(M)$, we have $\nabla^\alpha_X(Y\wedge\cdot)=(\nabla^M_X Y)\wedge\cdot+Y\wedge\nabla^\alpha_X\cdot$, which can be proven by a straightforward computation (the same relation holds for the interior product), we can write at $x \in M$:
\begin{eqnarray*}
\Delta^\alpha&=&d^\alpha\delta^\alpha+\delta^\alpha d^\alpha \\
&=&-\sum_{j,k=1}^n e_k^*\wedge \nabla^\alpha_{e_k}( e_j\lrcorner\nabla^\alpha_{e_j})-\sum_{j,k=1}^n e_j\lrcorner \nabla^\alpha_{e_j}( e_k^*\wedge\nabla^\alpha_{e_k})\\
&=&-\sum_{j,k=1}^n e_k^*\wedge ( e_j\lrcorner\nabla^\alpha_{e_k}\nabla^\alpha_{e_j})-\sum_{j,k=1}^n e_j\lrcorner ( e_k^*\wedge\nabla^\alpha_{e_j}\nabla^\alpha_{e_k}) \\
&=&-\sum_{j,k=1}^n e_k^*\wedge ( e_j\lrcorner\nabla^\alpha_{e_k}\nabla^\alpha_{e_j}) -\sum_{j=1}^n \nabla^\alpha_{e_j}\nabla^\alpha_{e_j}+ \sum_{j,k=1}^n e_k^*\wedge ( e_j\lrcorner\nabla^\alpha_{e_j}\nabla^\alpha_{e_k}) \\
&=&-\sum_{j=1}^n \nabla^\alpha_{e_j}\nabla^\alpha_{e_j}+\sum_{j,k=1}^n e_k^*\wedge (e_j\lrcorner R^\alpha(e_j,e_k)),
\end{eqnarray*}
where in the fourth equality we used the relation
$$
X \lrcorner ( \beta \wedge \cdot) = (X \lrcorner \beta) \wedge \cdot + (-1)^{\deg \beta}
\beta \wedge (X \lrcorner \cdot),
$$
for any differential form $\beta$. This shows that \eqref{eq:bochnermagnetic} holds. To obtain \eqref{eq:deltaalphaforms}, we just combine Lemma \ref{lem:magneticbochner} with the Bochner-Weitzenb\"ock formula \eqref{eq:bochner} and the fact that at $x \in M$
\begin{eqnarray*}
(\nabla^\alpha)^*\nabla^\alpha&=&-\sum_{j=1}^n\nabla^\alpha_{e_j}\nabla^\alpha_{e_j}\\
&=&-\sum_{j=1}^n\nabla^M_{e_j}(\nabla^M_{e_j}+i\alpha(e_j))-i\sum_{j=1}^n\alpha(e_j)(\nabla^M_{e_j}+i\alpha(e_j))\\
&=&\nabla^*\nabla+i\delta^M\alpha-2i\nabla^M_\alpha+|\alpha|^2.
\end{eqnarray*}
\end{proof}

\begin{remark}
Formula \eqref{eq:deltaalphaforms} is a generalization of the formula for the magnetic Laplacian for functions, given by
$$ \Delta^\alpha f = \delta^\alpha d ^\alpha f = \Delta^M f  + i (\delta^M \alpha) f - 2 i \alpha(f) + |\alpha|^2 f, $$
since $A^{[0],\alpha} = 0$.
\end{remark}

 Now, we will consider a particular case for the magnetic field $\alpha$. We will assume that it is a Killing 1-form, that is its corresponding vector field $\alpha^\sharp$ by the musical isomorphism is a Killing vector field. In this case, the standard Hodge Laplacian $\Delta^M$ commutes with $\mathcal{L}_\alpha$ since it commutes with all isometries. Indeed, we will show that, when $\alpha$ is of constant norm, the exterior differential $d^M$ and codifferential $\delta^M$ both commute with the magnetic Laplacian. Notice here that, in general, $d^\alpha$ and $\delta^\alpha$ do not commute with $\Delta^\alpha$ as a consequence of 
\eqref{eq:dalpha2} and even when $\alpha^\sharp$ is Killing.
We now show that Equation \eqref{eq:deltaalphaforms}
has the simpler expression \eqref{eq:relationkilling} in this case.
We also recall that for simplicity  $\alpha$ and $\alpha^\sharp$ are identified throughout the paper. 

\begin{proposition} \label{prop:deltaalphad} Let $(M^n,g)$ be a Riemannian manifold and let $\alpha$ be a Killing $1$-form, then
\begin{equation}\label{eq:relationkilling}
\Delta^\alpha=\Delta^M-2i\mathcal{L}_\alpha+|\alpha|^2,
\end{equation}
where $\mathcal{L}_\alpha$ is the Lie derivative in the direction of $\alpha$. In particular, $\mathcal{L}_\alpha\Delta^\alpha=\Delta^\alpha \mathcal{L}_\alpha$. Moreover, if the norm of $\alpha$ is constant, we have that $\Delta^\alpha d^M=d^M\Delta^\alpha$ and $\Delta^\alpha \delta^M=\delta^M\Delta^\alpha$ and, therefore, the magnetic Laplacian preserves the set of exact and co-exact forms.
\end{proposition}
\begin{proof}
The fact that $\alpha$ is Killing gives  $A^\alpha(X)=X\lrcorner d^M\alpha=2\nabla^M_X\alpha$ for any vector field $X\in TM$. Therefore, we get by \eqref{eq:extensionexpression} that
$$A^{[p],\alpha}=\sum_{j=1}^n e_j^*\wedge A^\alpha(e_j)\lrcorner=2\sum_{j=1}^n e_j^*\wedge \nabla^M_{e_j}\alpha\lrcorner=2T^{[p],\alpha},$$
where $T^{[p],X}$ is the canonical extension of the endomorphism $T^X=\nabla^M X$, for any $X$, given by the expression in \eqref{eq:extensionexpression}. Now, the identity  $\mathcal{L}_X=\nabla^M_X+T^{[p],X}$  valid on $p$-forms for any vector field $X$ on $TM$ \cite[Lem. 2.1]{S:09} allows us to deduce that
\begin{equation}\label{eq:liederivative}
2\mathcal{L}_\alpha=2\nabla^M_\alpha+A^{[p],\alpha}.
\end{equation}
Hence,  Equation \eqref{eq:deltaalphaforms} and the fact that $\delta^M \alpha=0$ since $\alpha$ is Killing gives the desired identity \eqref{eq:relationkilling}. 
In order to prove that $\mathcal{L}_\alpha$ commutes with $\Delta^\alpha$, we first use $\alpha(|\alpha|^2)=2g(\nabla^M_\alpha\alpha,\alpha)=0$ which is a consequence of the fact that $\alpha$ is Killing. Now, we compute, for any $p$-form $\omega$,
$$\mathcal{L}_\alpha(|\alpha|^2\cdot \omega)=\alpha(|\alpha|^2)\cdot \omega+|\alpha|^2\cdot \mathcal{L}_\alpha\omega=|\alpha|^2\cdot \mathcal{L}_\alpha\omega.$$
Thus, by the fact that $\mathcal{L}_\alpha$ commutes with the Laplacian $\Delta^M$, we get that $\mathcal{L}_\alpha\Delta^\alpha=\Delta^\alpha\mathcal{L}_\alpha$. Now we assume $\vert \alpha \vert$ is constant. It follows from Cartan's formula $\mathcal{L}_X \omega = X \lrcorner d^M \omega + d^M(X \lrcorner \omega)$ that $\mathcal{L}_X$ commutes with $d^M$ for any vector field $X$.
Since $d^M$ commutes with $\Delta^M$ and with $\mathcal{L}_\alpha$ as well as with multiplication by the constant $|\alpha|^2$, we deduce that $d^M$ commutes with $\Delta^\alpha$. That the codifferential  $\delta^M$ commutes with $\Delta^\alpha$ comes from the fact that $\delta^M$ commutes with $\Delta^M$ and with $\mathcal{L}_\alpha,$ which is a consequence of $\delta^M=\pm *d^M*$ and $\mathcal{L}_\alpha *=*\mathcal{L}_\alpha$ by Equation \eqref{eq:liederivative}. (Recall here that $A^{[p],\alpha}*=*A^{[n-p],\alpha}$). This finishes the proof.
\end{proof}

\begin{remark}
The relation \eqref{eq:liederivative} shows that for any complex differential forms $\omega$ and $\omega'$ on $M$, the following relation 
\begin{equation}\label{eq:adjointlie}
\langle\mathcal{L}_\alpha\omega,\omega'\rangle+\langle\omega,\mathcal{L}_\alpha\omega'\rangle=\alpha(\langle \omega,\omega'\rangle),
\end{equation}
holds pointwise when $\alpha$ is a Killing vector field (not necessarily of constant norm), since $A^{[p],\alpha}$ is skew-symmetric.
\end{remark}

When the magnetic potential $\alpha$ is Killing of constant norm on $(M^n,g)$, we have seen that the magnetic Laplacian $\Delta^{\alpha}$ preserves the set of exact and co-exact forms on $M$. In the following, we will assume $M$ to be compact and  will let $\lambda_{1,p}^{\alpha}(M)$ be the first non-negative eigenvalue of $\Delta^\alpha$ on differential $p$-forms  and $\lambda_{1,p}^{\alpha}(M)'$ (resp. $\lambda_{1,p}^{\alpha}(M)''$) be the first non-negative eigenvalue restricted to exact (resp. co-exact) $p$-forms. As in the standard case \cite{RS:11}, we can prove by  Hodge duality that $\lambda_{1,p}^{\alpha}( M)''=\lambda_{1,n-p}^{\alpha}(M)'$ and that 
$\lambda_{1,p}^{\alpha}(M)\leq {\rm min}(\lambda_{1,p}^{\alpha}( M)',\lambda_{1,p}^{\alpha}(M)'')$. Recall here that the magnetic Laplacian commutes with the Hodge star operator. However, we will see in the next proposition, that the relation  $\lambda_{1,p}^{\alpha}(M)={\rm min}(\lambda_{1,p}^{\alpha}( M)',\lambda_{1,p}^{\alpha}(M)'')$ that usually holds for the Laplacian $\Delta^M$ is not always true for $\Delta^\alpha$.

For the next proposition, we need the following well known result, which we present for completeness.

\begin{lemma}\label{harm}
Let $(M^n,g)$ be a compact manifold and let $X$ be a Killing vector field on $M$. For any harmonic form $\omega \in \Omega(M)$ we have
$$ \mathcal{L}_X \omega = 0. $$
\end{lemma}

\begin{proof}
  Let $\omega \in \Omega(M)$ be harmonic. Using Cartan's formula,
  we see that $\mathcal{L}_X \omega$ is exact. Moreover, since the Lie derivative of a Killing vector field commutes both with $d^M$ and $\delta^M$, the Lie derivative $\mathcal{L}_X\omega$ is both exact and harmonic. Therefore, by Hodge decomposition, ${\mathcal{L}}_X \omega = 0$.
\end{proof}

\begin{proposition}  Let $(M^n,g)$ be a compact Riemannian manifold and let $\alpha$ be a Killing $1$-form. The first non-negative eigenvalue $\lambda_{1,p}^{\alpha}(M)$ satisfies $\lambda_{1,p}^{\alpha}(M)=|\alpha|^2$ or $\lambda_{1,p}^{\alpha}(M)={\rm min}(\lambda_{1,p}^{\alpha}( M)',\lambda_{1,p}^{\alpha}(M)'')$ if $\alpha$ has constant norm. If $H^p(M)\neq 0$, then we get the estimate
$$\lambda^\alpha_{1,p}(M)\leq ||\alpha||^2_\infty.$$
\end{proposition}

\begin{proof}
Let $\omega$ be a complex $p$-eigenform of the magnetic Hodge Laplacian associated to the first eigenvalue $\lambda_{1,p}^{\alpha}(M)$. By the Hodge decomposition, we write $$\omega=d^M\omega_0+\delta^M\omega_1+\omega_2,$$
where $\omega_0\in \Omega^{p-1}(M,\mathbb{C}), \omega_1\in \Omega^{p+1}(M,\mathbb{C})$ and $\omega_2\in \Omega^{p}(M,\mathbb{C})$ is harmonic. 
From the equation $\Delta^\alpha\omega=\lambda_{1,p}^{\alpha}(M)\omega$, by uniqueness of the decomposition and the fact that both $d^M$ and $\delta^M$ commute with $\Delta^\alpha$, we obtain the relation
$\Delta^\alpha\omega_2=\lambda_{1,p}^{\alpha}(M)\omega_2$. Now, if $\omega_2$ does not vanish, then by the fact that $\alpha$ is Killing and $\omega_2$ is harmonic, we have by Lemma \ref{harm} that $\mathcal{L}_\alpha\omega_2=0$. Thus, by Equation \eqref{eq:relationkilling}, we get that $\Delta^\alpha\omega_2=|\alpha|^2\omega_2$ and, therefore, $\lambda_{1,p}^{\alpha}(M)=|\alpha|^2$. If $\omega_2$ vanishes, then we have $\omega=d^M\omega_0+\delta^M\omega_1$ and, hence, ${\rm min}(\lambda_{1,p}^{\alpha}( M)',\lambda_{1,p}^{\alpha}(M)'')\leq \lambda_{1,p}^{\alpha}(M)$. When $H^p(M)\neq 0$ then there is a non-vanishing $p$-harmonic form $\omega$ on $M$ and thus, as before, $\Delta^\alpha\omega=|\alpha|^2\omega$. Thus, by the min-max principle we deduce the required estimate. This finishes the proof.
\end{proof}

\begin{example}
As in the previous examples, consider the manifold $M=\mathbb{S}^3$ equipped with the standard metric of curvatue $1$. Let $Y_2$ be the unit Killing vector field as in Appendix \ref{sec:dudvalphaeig}. It follows that the $1$-forms $d^Mu$, $d^Mv$ and $\alpha = t Y_2$ are all simultaneous eigenforms of
  the operators $\Delta^\alpha$ such that
  \begin{eqnarray*}
  \Delta^\alpha d^Mu &=& (3+2t+t^2)d^Mu, \\
  \Delta^\alpha d^Mv &=& (3-2t+t^2)d^Mv, \\
  \Delta^\alpha \alpha &=& (4+t^2) \alpha.
  \end{eqnarray*}
  Moreover, $d^Mu, d^Mv$ are exact eigenforms associated to the smallest eigenvalue $\lambda_{1,1}'(M) =3$ and $\alpha$ is a co-exact eigenform associated to the smallest eigenvalue $\lambda_{1,1}''(M) = 4$ (see \cite{Paq79}). Therefore, we have for small $t > 0$,
  $$ \lambda_{1,1}^{\alpha}(M) = {\rm min}(\lambda_{1,1}^{\alpha}( M)',\lambda_{1,1}^{\alpha}(M)'') = 3-2t+t^2, $$
  since $H^1(M) = 0$. On the other hand, we get by Equation \eqref{eq:specmagS3} that for small $t>0$, $\lambda_{1,0}^\alpha(M) = t^2=| \alpha|^2.$ However, we have that $${\rm min}(\lambda_{1,0}^{\alpha}( M)',\lambda_{1,0}^{\alpha}(M)'')=\lambda_{1,0}^{\alpha}(M)''=3-2t+t^2.$$
\end{example}

\subsection{Gauge invariance of the magnetic Hodge Laplacian} \label{subsec:gaugeinv}

Another consequence of the magnetic Bochner-Weitzenb\"ock formula \eqref{eq:deltaalphaforms} is the following result.

\begin{corollary} \label{cor:gaugeinv}
 Let $(M^n,g)$ be a Riemannian manifold and let $\alpha$ be a differential $1$-form on $M$. For any $\alpha_\tau=\frac{d^M\tau}{i\tau}\in  \mathfrak{B}_M$ for some $\tau\in C^\infty(M,\mathbb{S}^1)$, the magnetic Laplacians $\Delta^\alpha$ and $\Delta^{\alpha + \alpha_\tau}$ on $p$-forms are unitarily equivalent, meaning that
$$ \bar \tau \Delta^\alpha \tau = \Delta^{\alpha+\alpha_\tau}. $$
In particular, $\Delta^\alpha$ and $\Delta^{\alpha+\alpha_\tau}$ have the same spectrum on a closed oriented Riemannian manifold.
\end{corollary}
\begin{proof}
The proof relies mainly on the following identity.
For any $f\in C^\infty(M,\mathbb{C})$ and $\omega\in \Omega^p(M,\CC)$, we have
$$\Delta^M(f\omega)=f\Delta^M\omega+(\Delta^M f)\omega-2\nabla^M_{d^Mf}\omega.$$
Hence, for $f=\tau\in C^\infty(M,\mathbb{S}^1)$, we use Equation \eqref{eq:deltaalphaforms} to compute
\begin{eqnarray}\label{eq:unit}
\bar\tau\Delta^\alpha(\tau\omega)&=&\bar\tau\left(\Delta^M(\tau\omega)-i A^{[p],\alpha}(\tau\omega)+ i (\delta^M \alpha)(\tau\omega) - 2 i \nabla^M_\alpha(\tau\omega)+ |\alpha|^2\tau\omega\right)\nonumber\\
&=&\Delta^M\omega+\bar\tau(\Delta^M\tau)\omega-2\bar\tau\nabla^M_{d^M\tau}\omega-i A^{[p],\alpha}\omega+ i (\delta^M \alpha)\omega\nonumber\\
&&- 2 i \bar\tau\alpha(\tau)\omega-2i\nabla^M_\alpha\omega+ |\alpha|^2\omega.
\end{eqnarray}
Taking the divergence of $d^M\tau=i\tau\alpha_\tau$, we get that $$\Delta^M\tau=\delta^M(i\tau\alpha_\tau)=i\tau\delta^M\alpha_\tau+\tau|\alpha_\tau|^2.$$
Hence, Equation \eqref{eq:unit} reduces to
\begin{eqnarray*}
\bar\tau\Delta^\alpha(\tau\omega)&=&\Delta^M\omega+i(\delta^M\alpha_\tau)\omega+|\alpha_\tau|^2\omega-2i\nabla^M_{\alpha_\tau}\omega-i A^{[p],\alpha}\omega+ i (\delta^M \alpha)\omega\nonumber\\
&&+ 2\langle\alpha,\alpha_\tau\rangle\omega-2i\nabla^M_\alpha\omega+ |\alpha|^2\omega\\
&=&\Delta^M\omega-i A^{[p],\alpha+\alpha_\tau}\omega+i\delta^M(\alpha_\tau+\alpha)\omega-2i\nabla^M_{\alpha+\alpha_\tau}\omega+|\alpha+\alpha_\tau|^2\omega\\
&=&\Delta^{\alpha+\alpha_\tau}\omega.
\end{eqnarray*}
In the second equality, we used the fact that $A^\alpha=A^{\alpha+\alpha_\tau}$ since $\alpha_\tau$ is a closed form. This allows us to deduce the result.
\end{proof}
The gauge invariance of the magnetic Laplacian allows us to state a Shikegawa type result for differential forms. 

\begin{theorem} \label{thm:shiforms} Let $(M^n,g)$ be a compact Riemannian manifold and let $\alpha$ be a one-form on $M$. Assume that $M$ carries a non-zero parallel $p$-form $\omega_0$ on $M$. Then we have the following: 
 \begin{itemize}
     \item[(a)] If $\alpha \in \mathfrak{B}_M$, then $\lambda_{1,p}^\alpha(M) = 0$ and there exists an eigenform $\omega$ of $\Delta^\alpha$ associated with the eigenvalue $\lambda_{1,p}^\alpha(M)$ such that     $f:=\langle \omega,\omega_0\rangle$ is nowhere vanishing.
     \item[(b)] Conversely, assume that $\alpha$ is Killing. If $\lambda_{1,p}^\alpha(M) = 0$ and there exists an eigenform $\omega$ of $\Delta^\alpha$ associated with the eigenvalue $\lambda_{1,p}^\alpha(M)$ such that $f:=\langle \omega,\omega_0\rangle$ is not vanishing, then $\alpha \in \mathfrak{B}_M$ and, in this case, it is a parallel form. 
 \end{itemize}
\end{theorem} 
\begin{proof}
 We first prove $(a)$. Since $\alpha=\frac{d^M\tau}{i\tau}\in  \mathfrak{B}_M$ for some $\tau\in C^\infty(M,\mathbb{S}^1)$, we deduce from Corollary \ref{cor:gaugeinv}  that the magnetic Laplacian has the same spectrum as the Hodge Laplacian $\Delta^M$. Hence the first eigenvalue $\lambda_{1,p}^\alpha(M)$ is equal to $0$ due to the existence of a parallel form $\omega_0$ which gives that $d^M \omega_0 = \delta^M \omega_0 = 0$. Moreover, one can easily check that the form $\omega:=\overline{\tau}\omega_0$ satisfies 
\begin{eqnarray*}
d^\alpha\omega&=&d^M(\overline{\tau}\omega_0)+i\alpha\wedge \overline{\tau}\omega_0\\
&=&d^M\overline{\tau}\wedge \omega_0+i\overline{\tau}\alpha\wedge \omega_0\\
&=&-i\overline{\tau}\alpha\wedge\omega_0+i\overline{\tau}\alpha\wedge \omega_0\\
&=&0.
\end{eqnarray*}
In the same way, we prove that $\delta^\alpha\omega=0$. Therefore, we have $\Delta^\alpha \omega = 0$. Hence the function $f=\langle\omega,\omega_0\rangle=\overline{\tau}|\omega_0|^2$ is nowhere zero since $\overline{\tau}\in \mathbb{S}^1$ and the parallel form $\omega_0$ is of constant norm. 

Now, we prove $(b)$. For this, we assume that $\alpha$ is Killing and we compute the Laplacian of the function $f$. We choose a local orthonormal frame $\{e_i\}$ of $TM$ such that $\nabla^M e_i|_x=0$ at some point $x$. Since the form $\omega_0$ is parallel, we write 
\begin{eqnarray*}
\Delta^M f(x)&=&-\sum_{i=1}^n (e_i(e_i(f)))(x)
=-\sum_{i=1}^n\langle\nabla^M_{e_i}\nabla^M_{e_i}\omega,\omega_0\rangle_x\\
&=&\langle\nabla^*\nabla\omega,\omega_0\rangle_x\\
&\stackrel{\eqref{eq:bochner}}{=}&\langle\Delta^M\omega-\mathcal{B}^{[p]}\omega,\omega_0\rangle_x\\
&\stackrel{\eqref{eq:relationkilling}}{=}&\langle\Delta^\alpha\omega+2i\mathcal{L}_\alpha\omega-|\alpha|^2\omega,\omega_0\rangle_x-\langle\omega,\mathcal{B}^{[p]}\omega_0\rangle_x\\
&\stackrel{\eqref{eq:adjointlie}}{=}&-2i\langle\omega,\mathcal{L}_\alpha\omega_0\rangle_x+2i\alpha(\langle\omega,\omega_0\rangle_x)-|\alpha(x)|^2 f(x)\\
&=&2i\alpha(f)(x) -|\alpha(x)|^2f(x).
\end{eqnarray*}
In this computation, we used the fact that $\mathcal{B}^{[p]}\omega_0=0$ since $\omega_0$ is parallel, and also that $\mathcal{L}_\alpha\omega_0=0$ by Lemma \ref{harm}. Therefore, Equation \eqref{eq:DalphaD} and ${\rm div} \alpha^\sharp =0$ (since $\alpha$ is Killing) allows us to deduce that $\Delta^\alpha f=0$ and, therefore $\lambda_1^\alpha(M)=0$. Now, the classical Shikegawa's result (Theorem \ref{thm:shi}) allows us to get that $\alpha \in \mathfrak{B}_M$ which is also equivalent to the fact that $d^M\alpha=0$ and $\int_c\alpha\in 2\pi\mathbb{Z}$ for all closed curves $c$ in $M$. Now, the condition $d^M\alpha=0$ means that $\nabla^M\alpha$ is a symmetric two-tensor which is also skew-symmetric by the fact that $\alpha$ is Killing. Hence, the form $\alpha$ is parallel.  
\end{proof}

\begin{remark}
We know from Lemma \ref{harm} that, on a compact manifold $(M^n,g)$, for any harmonic form $\omega$ and a Killing one-form $\alpha$, we have that $\mathcal{L}_\alpha\omega=0$. However, there are $\Delta^\alpha$-harmonic forms for which this fact no longer holds.
Indeed, assume that $M$ carries a Killing one-form $\alpha$ which is also in $\mathfrak{B}_M$, that is  $\alpha=\frac{d^M\tau}{i\tau}$ (for instance, such forms exist on the flat torus) and hence parallel by the same arguments as in the above proof. Assume also that a non-zero parallel $p$-form $\omega_0$ exists on $M$. We have seen from the proof of Theorem \ref{thm:shiforms} that $\omega=\overline{\tau}\omega_0$ is a $\Delta^\alpha$-harmonic form. Now, we compute 
$$\mathcal{L}_\alpha\omega=\alpha(\overline\tau)\omega_0+\overline{\tau}\mathcal{L}_\alpha\omega_0=-i\overline{\tau}|\alpha|^2\omega_0\neq 0,$$
since $\alpha$ is parallel and, hence, is of constant norm. 
\end{remark}

We illustrate Theorem \ref{thm:shiforms} with two examples.

\begin{examples}~
\begin{itemize}
\item[(a)] The flat torus $\mathbb{T}^n$ is trivialized by parallel $p$-forms for any $p$. Hence, one can always find, for any non-trivial differential form $\omega$, a parallel form $\omega_0$ such that $f=\langle\omega,\omega_0\rangle$ is not vanishing. Let $\alpha$ be any Killing one-form, we get 
\begin{equation}\label{eq:shitorus}
\lambda_{1,p}^\alpha(\mathbb{T}^n)=0 \Longleftrightarrow \alpha\in \mathcal{B}_{\mathbb{T}^n}.
\end{equation}
\item[(b)] Let us consider the product manifold $M=\mathbb{S}^1\times\mathbb{S}^3$ with the product metric. For $A\in \mathbb{R}$, we let $\alpha=A \omega_0$ be the one-form on $M$, where $\omega_0:=d\theta$ is the parallel unit one-form on $\mathbb{S}^1$. It is not difficult to check that $\alpha\in \mathcal{B}_{\mathbb{S}^1\times\mathbb{S}^3}$ if and only if $A\in \mathbb{Z}$. We show   
\begin{equation*}
\lambda_{1,1}^\alpha(\mathbb{S}^1\times\mathbb{S}^3)=0 \Longleftrightarrow A\in \mathbb{Z}.
\end{equation*}
When $A \in \mathbb{Z}$, the spectrum of $\Delta^\alpha$ is the same as the spectrum of $\Delta^M$, and hence $\lambda_{1,1}^\alpha(M)=0$ due to the existence of a parallel one-form. For the converse, assume that $\lambda_{1,1}^\alpha(\mathbb{S}^1\times\mathbb{S}^3)=0$ and that $A\notin\mathbb{Z}$. Hence, $\alpha\notin \mathcal{B}_{\mathbb{S}^1\times\mathbb{S}^3}$ and by Theorem \ref{thm:shiforms}, we obtain that $f=\langle\omega,\omega_0\rangle=0$ for any eigenform $\omega$ associated to $\lambda_{1,1}^\alpha(M)$. Therefore, if we consider an orthonormal frame on $\{\xi,e_1,e_2\}$ on $T\mathbb{S}^3$ such that $\xi$ is the unit Killing vector field that defines the Hopf fibration with $\nabla^{\mathbb{S}^3}_{e_1}\xi=e_2$ and $\nabla^{\mathbb{S}^3}_{e_2}\xi=-e_1$ (since the complex structure on $\mathbb{S}^2$ is given by $J(X) = \nabla_X^{\mathbb{S}^3} \xi$), we write 
$$\omega=f_0\xi+f_1 e_1+f_2 e_2$$
where $f_0,f_1,f_2$ are smooth functions on $\mathbb{S}^1\times\mathbb{S}^3$. Now, the condition $d^\alpha\omega=0$ allows us to get that $\frac{\partial f_k}{\partial\theta}=-iAf_k$ for $k=0,1,2$, which gives that $f_k=g_k e^{-iA\theta}$ with functions $g_k$ which are constant on $\mathbb{S}^1$. However, the functions $f_k$ are only periodic functions on $\mathbb{S}^1$ when $A\in \mathbb{Z}$, which is a contradiction.  
\end{itemize}
\end{examples}

\section{Eigenvalue estimates for the magnetic Hodge Laplacian on closed manifolds}

In this section, we establish several eigenvalue estimates for the magnetic Hodge Laplacian on a closed oriented Riemannian manifold $(M^n,g)$. In particular, we show that the diamagnetic inequality cannot hold in general.

\subsection{A magnetic Gallot-Meyer estimate}

The aim of this subsection is to derive a lower bound for the first eigenvalue of the magnetic Hodge Laplacian on $p$-forms that is analogous to that of Gallot-Meyer. We begin with the following lemma similar to \cite[Lem. 6.8]{GM:75}, relating the magnetic connection to the magnetic differential and co-differential.

\begin{lemma}\label{lem:gallotmeyer}
Let $(M^n,g)$ be a Riemannian manifold and let $\alpha$ be a magnetic potential. For any complex differential $p$-form $\omega$ with $p\geq 1$, we have
\begin{equation}\label{eq:twistor}
|\nabla^\alpha\omega|^2\geq \frac{1}{p+1}|d^\alpha\omega|^2+\frac{1}{n-p+1}|\delta^\alpha\omega|^2.
\end{equation}
\end{lemma}
\begin{proof}
The proof relies on defining the magnetic twistor form as in the usual case: For any  complex $p$-form $\omega$ and vector field $X\in \cX_\CC(M)$, we define
$$P^\alpha_X\omega:=\nabla_X^\alpha\omega-\frac{1}{p+1}X\lrcorner d^\alpha\omega+\frac{1}{n-p+1}X\wedge\delta^\alpha\omega.$$
Using Equation \eqref{eq:localddelta}, the norm of $P^\alpha$ is equal to
$$|P^\alpha\omega|^2:= \sum_{j=1}^n |P^\alpha_{e_j}\omega|^2 =  |\nabla^\alpha\omega|^2-\frac{1}{p+1}|d^\alpha\omega|^2-\frac{1}{n-p+1}|\delta^\alpha\omega|^2\geq 0.$$
Here we use the fact that any complex $p$-form $\beta$ on $M$ can be written as $\beta=\frac{1}{p}\sum_{j=1}^n e_j^*\wedge (e_j\lrcorner\beta)$, and therefore, $\sum_{j=1}^n |e_j\lrcorner\beta|^2=p|\beta|^2$ and $\sum_{j=1}^n |e_j^*\wedge\beta|^2=(n-p)|\beta|^2$.
\end{proof}

Applying Inequality \eqref{eq:twistor} to the $1$-form $\omega:=d^\alpha f$, where $f$ is a smooth complex-valued function, we get that
$$|{\rm Hess}^\alpha f|^2=|\nabla^\alpha d^\alpha f|^2\geq \frac{1}{2}|(d^\alpha)^2 f|^2+\frac{1}{n}|\Delta^\alpha f|^2\geq \frac{1}{n}|\Delta^\alpha f|^2.$$
If the equality is attained, then $(d^\alpha)^2f=0$ which, by \eqref{eq:dalpha2}, is equivalent to $d^M \alpha=0$. Therefore if equality occurs in \eqref{eq:Lich} (that is, if $\lambda_1^\alpha(M) = a_{-}(C,A,n)$), then from \cite[p. 1147]{{ELMP:16}}, we should have equality in the above inequality which means that necessarily   $d^M \alpha=0$. This explains why sharpness of the upper bound for $\lambda_1^\alpha(M)$ in \eqref{eq:Lich} is lost. The next result now reads as a ``magnetic version'' of the Gallot-Meyer estimate \cite[Thm. 6.13]{GM:75}.

\begin{theorem} \label{thm:gm}
Let $(M^n,g)$ be a closed oriented Riemannian manifold, and let $\alpha$ be a smooth $1$-form on $M$. Assume that $\mathcal{B}^{[p],\alpha}\geq K$ for some $K>0$ and $p\geq 1$. Then, we have
$$\lambda^\alpha_{1,p}(M)\geq \frac{C}{C-1}K,$$
where $C={\rm max}(p+1,n-p+1)$.
\end{theorem}

\begin{proof}
Let $\omega$ be a $p$-eigenform of $\Delta^\alpha$ associated to the first eigenvalue $\lambda^\alpha_{1,p}(M)$. We apply the magnetic Bochner formula to $\omega$, integrate it over $M$ and use inequality \eqref{eq:twistor} to obtain
\begin{eqnarray*}
 \lambda^\alpha_{1,p}(M)\int_M|\omega|^2d\mu_g&=&\int_M|\nabla^\alpha\omega|^2 d\mu_g+\int_M\langle \mathcal{B}^{[p],\alpha}\omega,\omega\rangle d\mu_g\\
 &\geq&\frac{1}{C}\int_M (|d^\alpha\omega|^2+|\delta^\alpha\omega|^2) d\mu_g+K\int_M|\omega|^2 d\mu_g\\
 &=& \left(\frac{\lambda^\alpha_{1,p}(M)}{C}+K\right)\int_M|\omega|^2 d\mu_g,
 \end{eqnarray*}
from which we deduce the desired inequality.
\end{proof}

\begin{remark} In view of Equality \eqref{eq:bochneroperatorstar} and since the Hodge star operator commutes with the magnetic Laplacian $\Delta^\alpha$ by Corollary \ref{cor:maglapstar}, it is enough to consider $p\leq \frac{n}{2}$ in the above estimate.
\end{remark}

\begin{example}
In order to check whether the condition $\mathcal{B}^{[p],\alpha}\geq K$  required in the previous theorem can be satisfied for some $K>0$, we will employ 
the example of the round sphere $\mathbb{S}^n$ for some odd $n=2m+1$ where the magnetic field $\alpha$ is given by $\alpha=t\xi$, for $t>0$, and $\xi$ is the unit Killing vector field on $\mathbb{S}^n$ that defines the Hopf fibration. Indeed, since on the round sphere $\mathcal{B}^{[p]}=p(n-p)$, we get that $\mathcal{B}^{[p],\alpha}=p(n-p)-tiA^{[p],\xi}$. Now, as $A^\xi X=X\lrcorner d^M\xi=2\nabla^M_X\xi$ for any vector field $X$, we can always find an orthonormal basis of $T\mathbb{S}^n$ such that the matrix of $A^\xi$ consists of the eigenvalue $0$ and block matrices of type $\begin{pmatrix}0&\pm 2\\ \mp 2&0\end{pmatrix}$. The eigenvalue $0$ corresponds to the eigenvector $\xi$ and the block matrices come from the fact that $\nabla^M\xi$ is the complex structure on $\xi^\perp$. Hence, in this basis, the eigenvalues of the symmetric matrix $iA^\xi$ are $-2,0,2$ with multiplicities $\frac{n-1}{2}, 1,\frac{n-1}{2}$ respectively. An easy computation shows that the $p$-eigenvalues of the matrix $iA^\xi$ are equal to
\begin{equation*}
\sigma_p = \left\{
\begin{matrix}
	-2p, & \text{if $p \le \frac{n-1}{2}$,}\\
	-2(n-p), & \text{if $p \ge \frac{n+1}{2}.$}
\end{matrix}\right.
\end{equation*}
Recall here that $n$ is odd. Hence the second inequality in \eqref{eq:upperbounda}  allows us to deduce that
\begin{equation*}
iA^{[p],\xi} \leq \left\{
\begin{matrix}
	2p, & \text{if $p \le \frac{n-1}{2}$,}\\
	2(n-p), & \text{if $p \ge \frac{n+1}{2}.$}
\end{matrix}\right.
\end{equation*}
Thus, for $t>0$, we deduce that
\begin{equation*}
\mathcal{B}^{[p],\alpha} \geq K=\left\{
\begin{matrix}
  p(n-p-2t), & \text{if $p \le \frac{n-1}{2}$,}\\
(p-2t)(n-p), & \text{if $p \ge \frac{n+1}{2}.$}
\end{matrix}\right.
\end{equation*}
Clearly, for any parameter $t\leq \frac{n-p}{2}$ or $\frac{p}{2}$, the number $K$ is positive. Hence, Theorem \ref{thm:gm} yields the following estimates for the first eigenvalue of the magnetic Laplacian $\Delta^\alpha$ on $\mathbb{S}^n$ with $\alpha=t\xi$,
\begin{equation*}
\lambda_{1,p}^\alpha(\mathbb{S}^n) \geq \left\{
\begin{matrix}
	\frac{n-p+1}{n-p}p(n-p-2t), & \text{if $p \le \frac{n-1}{2}$,}\\
	\frac{p+1}{p}(p-2t)(n-p), & \text{if $p \ge \frac{n+1}{2}.$}
\end{matrix}\right.
\end{equation*}
\end{example}

\subsection{A differential form analogue of a Colbois-El Soufi-Ilias-Savo estimate} \label{subsec:CESIS}

In 
\cite[Thm. 2]{CESIS-17}, the authors give an upper bound for the first Neumann eigenvalue of $\Delta^\alpha$ defined on complex functions in terms of some distance function of harmonic $1$-forms to a specific lattice and the norm of the magnetic field $d^M\alpha$ for Riemannian manifolds with boundary. In the following, we prove a similar result in the setting of differential forms for closed oriented Riemannian manifolds $(M^n,g)$. Before we state the result, let us first introduce some relevant notation: We denote by $m=b_1(M)$ the first Betti number and let $c_1,\dots,c_m$ be a basis of $H_1(M,\mathbb{Z})$ and $A_1,\dots,A_m \in H^1(M)$ be its dual basis, that is
$$ \frac{1}{2\pi}\int_{c_i} A_j = \delta_{ij}. $$
Let $\mathfrak{L}_{\mathbb{Z}}$ be the lattice
$$ \mathfrak{L}_{\mathbb{Z}} = \mathbb{Z} A_1 \oplus \mathbb{Z} A_2 \oplus \ldots \oplus \mathbb{Z} A_m. $$
If $H^1(M) = 0$ we set $\mathfrak{L}_\ZZ = 0$.  Note that, by Hodge Theory, we can think of $\mathfrak{L}_{\mathbb{Z}}$ as a discrete subset of all real harmonic $1$-forms.
We now introduce the following distance functions for any real $1$-form $\beta \in \Omega^1(M)$:
\begin{eqnarray*}
  d_2(\beta,\mathfrak{L}_{\mathbb{Z}}) &=& \sqrt{\inf_{\eta \in \mathfrak{L}_{\mathbb{Z}}} \Vert \beta - \eta \Vert_2^2}, \\
  d_\infty(\beta,\mathfrak{L}_{\mathbb{Z}}) &=& \sqrt{\inf_{\eta \in \mathfrak{L}_{\mathbb{Z}}} \Vert \beta - \eta \Vert_\infty^2}.
\end{eqnarray*}

When $\mathfrak{L}_\ZZ = 0$, the above distances reduce to $||\beta||_2$ or $||\beta||_\infty$. Now, we state the main result of this section.
\begin{theorem}\label{thm:CS}
  Let $(M^n,g)$ be a closed Riemannian manifold and $\alpha \in \Omega^1(M)$ be a magnetic potential of the form $\alpha = \delta^M \psi + h$ with $h$ a harmonic $1$-form and $\psi$ a $2$-form. Then we have the following eigenvalue estimate for the magnetic Hodge Laplacian on complex $p$-forms:
  \begin{equation} \label{eq:cesis}
  \lambda_{1,p}^\alpha(M) \le \lambda_{1,p}(M) + \min\left\{ d_\infty(\alpha,\mathfrak{L}_{\mathbb{Z}})^2,\frac{\Vert \omega_0 \Vert_\infty^2}{\Vert \omega_0 \Vert_{2}^2} d_2(\alpha,\mathfrak{L}_{\mathbb{Z}})^2 \right\}
  \end{equation}
  with
  \begin{equation} \label{eq:cesis2}
  d_2(\alpha,\mathfrak{L}_{\mathbb{Z}})^2 \le  d_2(h,\mathfrak{L}_\mathbb{Z})^2 + \frac{\Vert d^M\alpha \Vert_2^2}{\lambda''_{1,1}(M)},
  \end{equation}
 where $\omega_0$ is a real eigenform of the Hodge Laplacian $\Delta^M$ associated to the first eigenvalue $\lambda_{1,p}(M)$, and $\lambda''_{1,1}(M)$ denotes the first eigenvalue of the Hodge Laplacian on co-exact $1$-forms.
\end{theorem}

\begin{proof} The proof mainly follows the same lines as in  \cite{CESIS-17}. Firstly, we choose $\omega_0$ to be a real $p$-form. Let $\eta \in \mathfrak{L}_{\mathbb{Z}}$, that is
  $$ \eta = n_1 A_1 + n_2 A_2 + \ldots + n_m A_m \in \mathfrak{L}_{\mathbb{Z}}, $$
for some integers $n_1,\ldots, n_m \in \ZZ$.  We fix $x_0 \in M$ and define
  $$ u(x) = e^{i \int_{x_0}^x \eta}. $$
  The r.h.s. 
  is well defined and independent of the path from $x_0$ to $x$ chosen, since $\int_{x_0}^x \eta$ coincides for any pair of homotopic curves from $x_0$ and $x$ and agrees up to a multiple of $2\pi$ for any arbitrary pair of paths from $x_0$ to $x$ as $\eta \in \mathfrak{L}_{\mathbb{Z}}$. Then we have
  $d^M u = i u \eta$. Therefore, for the $p$-form $\omega:= u \omega_0$, we compute
  $$ d^\alpha \omega = d^M \omega + i \alpha \wedge \omega = (d^Mu) \wedge \omega_0 + u d^M\omega_0 + i u \alpha \wedge \omega_0 = u d^M \omega_0 + iu (\eta + \alpha) \wedge \omega_0. $$
  Similarly,
  $$ \delta^\alpha \omega = \delta^M \omega - i \alpha^\sharp \lrcorner \omega = u \delta^M \omega_0 - (d^Mu)^\sharp \lrcorner \omega_0 - i u \alpha^\sharp \lrcorner \omega_0 = u\delta^M \omega_0 - iu (\eta + \alpha)^\sharp \lrcorner \omega_0. $$
  Now we take the norms and use orthogonality of its real and imaginary parts to obtain
  $$ |d^\alpha \omega|^2 = |d^M\omega_0|^2 + |(\eta+\alpha) \wedge \omega_0|^2, $$
  and similarly
  $$|\delta^\alpha \omega|^2 = |\delta^M \omega_0|^2 +| (\eta+\alpha)^\sharp \lrcorner \omega_0|^2. $$
  Using the fact that $|X \wedge \omega |^2 + |X^\sharp\lrcorner \omega|^2 = |X|^2 \cdot |\omega|^2$ for any vector field $X$, we add the above two equations and choose $\omega_0$ to be an eigenform of the Hodge Laplacian to estimate
  \begin{eqnarray*}
   \lambda_{1,p}^\alpha(M) &\le& \frac{\int_M (|d^\alpha \omega|^2 +| \delta^\alpha \omega|^2)d\mu_g}{\int_M |\omega|^2d\mu_g} \\
   &=& \frac{\int_M (|d^M\omega_0|^2+|\delta^M \omega_0|^2)d\mu_g}{\int_M |\omega_0|^2d\mu_g} + \frac{\int_M |\eta + \alpha|^2 |\omega_0|^2d\mu_g}{\int_M |\omega_0|^2d\mu_g} \\
   &=& \lambda_{1,p}(M) + \frac{\int_M |\eta + \alpha|^2 |\omega_0|^2d\mu_g}{\int_M |\omega_0|^2d\mu_g}
  \end{eqnarray*}
  with
  $$ \frac{\int_M |\eta + \alpha|^2 |\omega_0|^2d\mu_g}{\int_M |\omega_0|^2d\mu_g} \le \min\left\{ \Vert \eta + \alpha \Vert_\infty^2, \frac{\Vert \omega_0 \Vert_\infty^2}{\Vert \omega_0 \Vert_2^2} \Vert \eta+\alpha \Vert_2^2
  \right\}.
  $$
  Since $\eta \in \mathfrak{L}_{\mathbb{Z}}$ was arbitrary, this proves Inequality \eqref{eq:cesis}.

  For the proof of Inequality \eqref{eq:cesis2}, recall that we have
  $\alpha = \delta^M \psi + h$. Since harmonic $1$-forms are $L^2$-orthogonal to the forms in $\delta^M(\Omega^2(M))$, we have
  $$ d_2(\alpha,\mathfrak{L}_{\mathbb{Z}})^2 = \inf_{\eta \in {\mathfrak{L}_{\mathbb{Z}}}} \int_M |\eta + \alpha|^2 d\mu_g 
  = \int_M |\delta^M \psi|^2 d\mu_g +  d_2(h,\mathfrak{L}_{\mathbb{Z}})^2. 
  $$
  Since $\delta^M \psi$ is co-exact, we have
  $$ \frac{\int_M |d^M \delta^M \psi|^2 d\mu_g }{\int_M |\delta^M\psi|^2 d\mu_g} \ge \lambda''_{1,1}(M),  $$
  and therefore, 
 \begin{eqnarray*}
  d_2(\alpha,\mathfrak{L}_{\mathbb{Z}})^2 &\le& \frac{\int_M |d^M \delta^M \psi|^2d\mu_g}{\lambda_{1,1}''(M)} +d_2(h,\mathfrak{L}_{\mathbb{Z}})^2 \\
  &=& \frac{\Vert d^M \alpha\Vert_2^2}{\lambda_{1,1}''(M)} + d_2(h,\mathfrak{L}_{\mathbb{Z}})^2.
  \end{eqnarray*}
  This finishes the proof of the theorem.
\end{proof}

\begin{remark}
  The factor $\frac{\Vert \omega_0 \Vert_\infty^2}{\Vert \omega_0 \Vert_2^2}$ requires knowledge of the $p$-eigenform
  of the smallest eigenvalue. Under certain curvature conditions, it can be estimated from above as explained in \cite{Li80}.
\end{remark}

\subsection{The diamagnetic inequality does not hold for the magnetic Hodge Laplacian}
\label{subsec:diamagineq}

A natural question is whether the diamagnetic inequality also holds for the magnetic Hodge Laplacian. That is, whether the inequality $$\lambda_{1,p}^\alpha(M) \geq \lambda_{1,p}(M)$$ holds or not for some $p\geq 1$. 
An example where the diamagnetic inequality holds is the flat $n$-dimensional torus $M=\mathbb{T}^n$. Clearly, the first eigenvalue $\lambda_{1,p}(M)=0$ for any $p$ due to the existence of a parallel $p$-form. Hence the inequality $\lambda_{1,p}^\alpha(M)\geq 0=\lambda_{1,p}(M)$ is satisfied. However, according to \eqref{eq:shitorus}, the first eigenvalue $\lambda_{1,p}^\alpha(M)$ can be positive.
In this subsection, we provide an example to show that the diamagnetic inequality does not hold in general. While this inequality is true for $p=0$, we provide a counterexample for $p=1$. We also give an explicit characterisation which determines whether this inequality holds for $\Delta^{t\xi}$ where $\xi$ is a Killing vector field. We start with the following estimate:

\begin{theorem} \label{thm:eigtaylor}
  Let $(M^n,g)$ be a closed oriented Riemannian manifold and $\xi \in \Omega^1(M)$. Then, for any $t \in \RR$, we have, for $\alpha=t\xi$,
  \begin{equation} \label{eq:eigvalcomp}
  \lambda^{\alpha}_{1,p}(M) \le \lambda_{1,p}(M) + \frac{2t}{ \Vert\omega\Vert_2^2} {\rm Im} \left( \int_M \langle \mathcal{L}_{\xi} \omega, \omega \rangle d\mu_g \right) + t^2 \Vert \xi \Vert_\infty^2,
  \end{equation}
  where $\omega \in \Omega^p(M,\CC)$ is an eigenform of the Hodge Laplacian $\Delta^M$ (linearly extended to complex $p$-forms) associated with the eigenvalue $\lambda_{1,p}(M)$, and $\mathcal{L}_X$ is the Lie derivative
 in the direction of the vector field $X \in \cX(M)$. In particular, if ${\rm Im} \left( \int_M \langle \mathcal{L}_{\xi} \omega, \omega \rangle d\mu_g\right)$ is negative for some complex eigenform $\omega$, then we get for small positive $t$ that
  $$ \lambda^{\alpha}_{1,p}(M) < \lambda_{1,p}(M),$$
which means that the diamagnetic inequality does not hold.
\end{theorem}

\begin{proof} Let $\omega$ be any $p$-form  in $\Omega^p(M,\CC)$. By the characterization of the first eigenvalue, we have for $\alpha=t\xi$
  $$\lambda^{\alpha}_{1,p}(M) \le \frac{\int_M (|d^{\alpha} \omega|^2 + |\delta^{\alpha} \omega|^2)d\mu_g}{\int_M |\omega|^2 d\mu_g} = \frac{\int_M (|d^M\omega + i t \xi \wedge \omega|^2 +|\delta^M \omega - i t \xi \lrcorner \omega|^2) d\mu_g}{\int_M |\omega|^2 d\mu_g}.$$
Now, we compute
$$ \int_M |d^M \omega + i t \xi \wedge \omega|^2 d\mu_g = \Vert d^M \omega \Vert_2^2 + 2t {\rm{Re}} \left( \int_M \langle d^M\omega, i \xi \wedge \omega \rangle d\mu_g\right) + t^2 \Vert \xi \wedge \omega \Vert_2^2 $$
and
$$ \int_M |\delta^M \omega - i t \xi \lrcorner \omega|^2 d\mu_g = \Vert \delta^M \omega \Vert_2^2 - 2t
{\rm{Re}} \left( \int_M \langle \delta^M \omega, i \xi \lrcorner \omega \rangle d\mu_g\right) + t^2 \Vert \xi \lrcorner \omega \Vert_2^2. $$
Adding both equations and using the Cartan formula $\mathcal{L}_X \omega=X \lrcorner d^M\omega + d^M(X \lrcorner \omega)$ for any vector field $X$ yields
  \begin{multline*}
     \int_M \left(|d^M \omega + i t \xi \wedge \omega|^2 +|\delta^M \omega - i t \xi \lrcorner \omega|^2\right)d\mu_g =\Vert d^M\omega\Vert_2^2+ \Vert \delta^M\omega\Vert_2^2\\-2t{\rm{Re}}\left(
     \int_M \left(\langle i \xi \lrcorner d^M\omega, \omega\rangle + \langle i d^M (\xi \lrcorner \omega), \omega \rangle\right) d\mu_g\right) + t^2 \int_M |\xi|^2 \cdot |\omega|^2 d\mu_g \\ = \Vert d^M\omega\Vert_2^2+ \Vert \delta^M\omega\Vert_2^2+ 2t {\rm{Im}} \left( \int_M
     \langle \mathcal{L}_\xi \omega, \omega \rangle d\mu_g\right) + t^2 \int_M |\xi|^2 \cdot |\omega|^2 d\mu_g.
  \end{multline*}
Choosing $\omega\in \Omega^p(M,\mathbb{C})$ to be an eigenform of $\Delta^M$ with respect to the eigenvalue $\lambda_{1,p}(M)$, we conclude that
  $$ \lambda_{1,p}^{\alpha}(M) \le \lambda_{1,p}(M) + \frac{2t}{\Vert \omega \Vert_2^2} {\rm{Im}} \left( \int_M \langle \mathcal{L}_\xi \omega, \omega \rangle d\mu_g\right) + t^2 \Vert \xi \Vert_\infty^2. $$
 This finishes the proof of the stated inequality. The last part is a direct consequence of the fact that when ${\rm Im} \left( \int_M \langle \mathcal{L}_{\xi} \omega, \omega \rangle d\mu_g\right)<0$  one can then always find positive small enough $t$  so that the r.h.s of the above inequality is strictly less than $ \lambda_{1,p}(M)$.
\end{proof}

\begin{remark}
  Note that the real and imaginary parts of a complex eigenform of $\Delta^M$ are both also eigenforms of $\Delta^M$ associated with the same eigenvalue. Therefore, in order to have
  $$  {\rm{Im}} \left( \int_M \langle \mathcal{L}_\xi \omega, \omega \rangle d\mu_g\right) \neq 0, $$
  the eigenspace $E_{\rm{min}}$ of $\Delta^M$ associated with the smallest eigenvalue $\lambda_{1,p}(M)$ needs to be at least $2$-dimensional. Of course, this higher dimensionality does not necessarily imply that this term is non-zero.
\end{remark}

In order to interpret the condition $ {\rm{Im}} \left( \int_M \langle \mathcal{L}_\xi \omega, \omega \rangle d\mu_g\right)<0$ in Theorem \ref{thm:eigtaylor}, we will consider the case when the vector field $\xi$ is Killing. 

\begin{corollary}\label{cor:diam}
    Let $(M^n,g)$ be a closed oriented Riemannian manifold, $\xi$ be a Killing vector field on $M$ and 
    $V:={\rm Ker}(\mathcal{L}_\xi|_{\Omega^p(M,\mathbb{C})})$ for some fixed $p$. If the eigenspace ${\rm Ker}(\Delta^M-\lambda_{1,p}(M){\rm Id})$ associated with the first eigenvalue $\lambda_{1,p}(M)$ is not included in $V$, 
    then the diamagnetic inequality does not hold. If the eigenspace ${\rm Ker}(\Delta^M-\lambda_{1,p}(M){\rm Id})$ is included in $V$, 
    then the diamagnetic inequality holds (at least for magnetic potentials $t \xi$ with small $|t| > 0$).
    \end{corollary}
  \begin{proof} 
The Laplacian $\Delta^M$ commutes with $\mathcal{L}_\xi$ and, thus, the Lie derivative $\mathcal{L}_\xi$ preserves the eigenspace ${\rm Ker}(\Delta^M-\lambda_{1,p}(M){\rm Id})$ which is of finite dimension. Relation \eqref{eq:adjointlie} says that the formal $L^2$-adjoint of $\mathcal{L}_\xi$ is equal to $-\mathcal{L}_\xi$. Hence, the matrix of $\mathcal{L}_\xi$ is skew-symmetric and, thus, the eigenvalues are of the form $0$ and $\pm i\beta$. Since by assumption the eigenspace ${\rm Ker}(\Delta^M-\lambda_{1,p}(M){\rm Id})$ is not in the kernel of $\mathcal{L}_\xi$, we can always find an eigenform $\omega$ of $\Delta^M$ such that $\mathcal{L}_\xi\omega =i\beta\omega$ for some $\beta$ with $\beta<0$ (if $\beta>0$, we choose its conjuguate $\overline\omega$). Hence, for such an eigenform, we deduce that $ {\rm{Im}} \left( \int_M \langle \mathcal{L}_\xi \omega, \omega \rangle d\mu_g\right)=\beta\int_M|\omega|^2 d\mu_g<0$. Therefore by Theorem \ref{thm:eigtaylor} the diamagnetic inequality does not hold for small positive $t$.  

To prove the second statement, we have from Relation \eqref{eq:relationkilling} that  $\Delta^\alpha=\Delta^M-2it\mathcal{L}_\xi+t^2|\xi|^2$ holds for $\alpha=t\xi$. Also, we know from Proposition \ref{prop:deltaalphad} that $\mathcal{L}_\xi\Delta^\alpha=\Delta^\alpha\mathcal{L}_\xi$. Therefore, the operator $\Delta^\alpha$ preserves the space $V$ as well as its orthogonal complement, by the fact that it is a self-adjoint operator. 
As $\Delta^M$ is perturbed analytically, the family $(\Delta^{t\xi})_t$ is an analytic family of self-adjoint operators with compact resolvent and therefore the Hilbert basis of $p$-eigenforms of $\Delta^M$ and their corresponding eigenvalues can be extended analytically in the perturbation parameter $t$ to a Hilbert basis of $p$-eigenforms of $\Delta^{t \xi}$ and their corresponding eigenvalues (see \cite[Thm. VII.3.9]{Ka95}). Since by assumption ${\rm Ker}(\Delta^M-\lambda_{1,p}(M){\rm Id})\subset V$ and the fact that the spectrum is discrete (with finite dimensional eigenspaces), we deduce that $\lambda_{1,p}^{0}(V)=\lambda_{1,p}(M)$, where $\lambda_{1,p}^{0}(V)$ denotes the lowest eigenvalue of $\Delta^M$ on $p$-forms, restricted to the invariant subspace $V$. Therefore, by choosing the analytic perturbation $\omega_t$ of any basis element $\omega$ in the $\Delta^M$-eigenspace corresponding to the eigenvalue $\lambda_{1,p}(M)$ and using the fact that $\omega_t$ is of unit $L^2$-norm, we obtain the estimate 
$$\int_M \langle \Delta^{t\xi}\omega_t,\omega_t\rangle d\mu_g=\int_M \langle \Delta^M\omega_t,\omega_t\rangle d\mu_g+t^2\int_M|\xi|^2|\omega_t|^2 d\mu_g \ge \lambda_{1,p}(M). $$
In the last inequality, we use the min-max principle for $\Delta^M$.
This implies that $\lambda^{t\xi}_{1,p}(V) \ge \lambda_{1,p}(M)$ for all $t$. The continuity of the maps $t\mapsto \lambda_{j,p}^{t\xi}(V)$ and $t\mapsto \lambda_{j,p}^{t\xi}(M)$ along with the fact that the eigenvalues are discrete and $\lambda_{1,p}^{0}(V)=\lambda_{1,p}(M)$ imply that $\lambda_{1,p}^{t\xi}(V)=\lambda_{1,p}^{t\xi}(M)$ for small $t$.  Hence, we deduce that $\lambda_{1,p}^{t\xi}(M)\geq \lambda_{1,p}(M)$ for small $|t|$.

\end{proof}

\color{black}

Below we consider the $3$-dimensional round sphere and show that the diamagnetic inequality is not satisfied for a suitable choice of magnetic potential. For more details on the computation, we refer to Appendix \ref{sec:berger}.
\begin{corollary} \label{cor:notdiamagineq}
  Let $(M=\mathbb{S}^3,g)$ be the $3$-dimensional unit sphere (centered at the origin) equipped with the canonical Riemannian metric $g$ of curvature $1$. Let $\xi=Y_2$ be the unit Killing vector field on $\mathbb{S}^3$ as in Appendix \ref{sec:berger}. Then, for small $t > 0$, we have, for $\alpha=tY_2$,
  $$ \lambda_{1,1}^{\alpha}(M) < \lambda_{1,1}(M),$$
 which means that the diamagnetic inequality does not hold in general for differential $1$-forms.
\end{corollary}

 \begin{proof} 
 Hence, from Corollary \ref{cor:diam}, we just need to find a $1$-eigenform of the Laplacian $\Delta^M$ which is not in the kernel of $\mathcal{L}_\xi$. 
 For this, we use the computations done in Appendix \ref{sec:dudvalphaeig}. Let $(a,b), (z_1,z_2) \in \mathbb{C}^2 \backslash (0,0)$ and set
  $$ v(z_1,z_2) = b\bar z_1 - a \bar z_2. $$
Recall that $\Delta^M v = 3 v$ and that $3$ is the smallest eigenvalue of $\Delta^M$ associated to the $1$-form $\omega:=d^M v$. Hence, we compute
$$\mathcal{L}_{\xi} d^Mv=d^M(\mathcal{L}_{\xi}v)=-id^M v.$$
In the last equality, we use the following consequence of the identity \eqref{eq:Y2phi}:
$ \mathcal{L}_{\xi}v = Y_2(v)= - i v.$
Hence the result follows from Corollary \ref{cor:diam}. 
\end{proof}
 
\section{The magnetic Hodge Laplacian on manifolds with boundary}

\subsection{A magnetic Green's formula for differential forms}
\label{subsec:greensformula}

Let $(M^n,g)$ be a compact oriented Riemannian manifold with smooth boundary $\partial M$ and let $\alpha \in \Omega^1(M)$. We denote by $\nu$ the unit inward normal vector field to $\partial M$ and by $\iota:\partial M\to M$ the canonical injection. For any pair of complex differential forms $\omega_1$ and $\omega_2$, the magnetic Stokes formula
$$\int_M \langle d^\alpha\omega_1,\omega_2\rangle d\mu_g=\int_M\langle\omega_1,\delta^\alpha\omega_2\rangle d\mu_g-\int_{\partial M}\langle \iota^*\omega_1,\nu\lrcorner\omega_2\rangle d\mu_g$$
holds. Here $\iota^*$ is the pull-back of differential forms on $M$ to the boundary. Indeed, it can be deduced from the usual Stokes formula and the expression of $d^\alpha$ and $\delta^\alpha$.
As a consequence, we get
\begin{eqnarray}\label{eq:productlaplacian}
\int_M\langle\Delta^\alpha\omega_1,\omega_2\rangle d\mu_g&=&\int_M \langle d^\alpha\delta^\alpha\omega_1+\delta^\alpha d^\alpha\omega_1,\omega_2\rangle d\mu_g\nonumber\\
&=&\int_M\langle\delta^\alpha\omega_1, \delta^\alpha\omega_2\rangle d\mu_g-\int_{\partial M}\langle\iota^*(\delta^\alpha\omega_1),\nu\lrcorner\omega_2\rangle d\mu_g\nonumber\\&&+\int_M\langle d^\alpha\omega_1,d^\alpha\omega_2\rangle d\mu_g+\int_{\partial M}\langle\nu\lrcorner d^\alpha\omega_1,\iota^*\omega_2\rangle d\mu_g\\
&=&\int_M\langle\omega_1,\Delta^\alpha\omega_2\rangle d\mu_g+\int_{\partial M}\langle\nu\lrcorner\omega_1,\iota^*(\delta^\alpha\omega_2)\rangle d\mu_g\nonumber\\
&&-\int_{\partial M}\langle\iota^*(\delta^\alpha\omega_1),\nu\lrcorner\omega_2\rangle d\mu_g-\int_{\partial M}\langle\iota^*\omega_1,\nu\lrcorner d^\alpha\omega_2\rangle d\mu_g\nonumber\\&&+\int_{\partial M}\langle\nu\lrcorner d^\alpha\omega_1,\iota^*\omega_2\rangle d\mu_g. \nonumber
\end{eqnarray}

Hence, we deduce that the magnetic Laplacian on smooth differential forms with Dirichlet boundary condition is self-adjoint and, being elliptic, it has a discrete spectrum that consists of real nonnegative eigenvalues.

\subsection{A magnetic Reilly formula}

In the following, we establish a Reilly formula for the magnetic Hodge Laplacian on a compact oriented Riemannian manifold $(M^n,g)$ with smooth boundary $\partial M$ as in \cite[Thm. 3]{RS:11}. (Note that the dimension of the manifold in \cite{RS:11} is $n+1$ in contrast to our setting).
 \begin{theorem} \label{thm:reilly}
Let $(M^n,g)$ be a compact oriented Riemannian manifold with smooth boundary $\partial M$ and let $\alpha \in \Omega^1(M)$.  Then we have for any $\omega \in \Omega^p(M,\mathbb{C})$, $p \ge 1$, the magnetic Reilly formula
\begin{multline*}
    \int_M (|d^\alpha \omega|^2 +|\delta^\alpha \omega|^2)d\mu_g = \int_M |\nabla^\alpha \omega|^2 d\mu_g+ \int_M\langle \mathcal{B}^{[p],\alpha}\omega,\omega \rangle d\mu_g\\ + 2  {\rm{Re}}\left(  \int_{\partial M} \langle d^{\alpha^T}(\nu \lrcorner \omega), \iota^*\omega \rangle d\mu_g\right) + \int_{\partial M} \langle \mathrm{II}^{[p]} \iota^* \omega,\iota^* \omega \rangle d\mu_g\\+ \int_{\partial M} \langle \mathrm{II}^{[n-p]} \iota^* (* \omega),\iota^* (* \omega) \rangle d\mu_g
\end{multline*}
where 
$\alpha^T=\iota^*\alpha \in \Omega^1(\partial M)$ is the tangential component of $\alpha$,
${\rm{II}} = -\nabla^M \nu$ is the Weingarten tensor of the boundary and
$d^{\alpha^T}:= d^{\partial M}+ i \alpha^T\wedge$. Here ${\rm II}^{[p]}$ is the extension of ${\rm II}$ as defined in  \eqref{eq:extension}.
\end{theorem}

\begin{proof}
The proof follows the same lines as in \cite[Thm. 3]{RS:11}. Indeed, we just need to integrate the magnetic Bochner-Weitzenb\"ock formula \eqref{eq:bochnermagnetic} over the manifold $M$. From Equation \eqref{eq:productlaplacian}, we have that
\begin{eqnarray*}
\int_M\langle\Delta^\alpha\omega,\omega\rangle d\mu_g
&=&\int_M|\delta^\alpha\omega|^2 d\mu_g-\int_{\partial M}\langle\iota^*(\delta^\alpha\omega),\nu\lrcorner\omega\rangle d\mu_g+\int_M|d^\alpha\omega|^2 d\mu_g\\&&+\int_{\partial M}\langle\nu\lrcorner d^\alpha\omega,\iota^*\omega\rangle d\mu_g.
\end{eqnarray*}
Notice here that $\int_M\langle\Delta^\alpha\omega,\omega\rangle d\mu_g$ is not necessarily real. Now from \cite[Lemma 18]{RS:11} and the expression of $\delta^\alpha$, one can easily deduce the following
$$\delta^{\alpha^T}(\iota^*\omega)=\iota^*(\delta^\alpha\omega)+\nu\lrcorner\nabla^\alpha_\nu\omega+{\rm{II}}^{[p-1]}(\nu\lrcorner\omega)-(n-1)H\nu\lrcorner\omega$$
where the mean curvature $H:= \frac{1}{n-1} {\rm{tr}} (\rm{II})$ of $\partial M \subset M$. Also, using the expression of $d^\alpha$, we have that,
$$d^{\alpha^T}(\nu\lrcorner\omega)=-\nu\lrcorner d^\alpha\omega+\iota^*(\nabla^\alpha_\nu\omega)-{\rm{II}}^{[p]}(\iota^*\omega).$$
Therefore, we arrive at
\begin{eqnarray*}
\int_M\langle\Delta^\alpha\omega,\omega\rangle d\mu_g&=&\int_M(|d^\alpha\omega|^2+|\delta^\alpha\omega|^2) d\mu_g\\
&-&\int_{\partial M} \langle \delta^{\alpha^T}(\iota^*\omega)-\nu\lrcorner\nabla^\alpha_\nu\omega-{\rm{II}}^{[p-1]}(\nu\lrcorner\omega)+(n-1)H\nu\lrcorner\omega,\nu\lrcorner\omega\rangle d\mu_g\\
&+&\int_{\partial M}\langle -d^{\alpha^T}(\nu\lrcorner\omega)+\iota^*(\nabla^\alpha_\nu\omega)-{\rm{II}}^{[p]}(\iota^*\omega),\iota^*\omega\rangle d\mu_g\\
&=&\int_M(|d^\alpha\omega|^2+|\delta^\alpha\omega|^2) d\mu_g-2 {\rm{Re}}\left( \int_{\partial M} \langle d^{\alpha^T}(\nu\lrcorner\omega),\iota^*\omega\rangle d\mu_g \right)\\&+&\int_{\partial M}\langle\nabla^\alpha_\nu\omega,\omega\rangle d\mu_g+\int_{\partial M}\langle{\rm{II}}^{[p-1]}\nu\lrcorner\omega,\nu\lrcorner\omega\rangle d\mu_g\\&-&\int_{\partial M}(n-1)H|\nu\lrcorner\omega|^2 d\mu_g-\int_{\partial M}\langle{\rm{II}}^{[p]}\iota^*\omega,\iota^*\omega\rangle d\mu_g.
\end{eqnarray*}
In the above equality, we use the fact that $\nabla^\alpha_\nu\omega=\iota^*(\nabla^\alpha_\nu\omega)+\nu\wedge(\nu\lrcorner\nabla^\alpha_\nu\omega)$ at any point on the boundary. Now since the identity $*_{\partial M}{\rm{II}}^{[p-1]}+{\rm{II}}^{[n-p]}*_{\partial M}=(n-1)H*_{\partial M}$ holds on $(p-1)$-forms on $\partial M$ \cite{RS:11}, we apply it to the form $\nu\lrcorner\omega$ and take the Hermitian product with $*_{\partial M}(\nu\lrcorner\omega)$. This leads to the following
$$\langle {\rm{II}}^{[p-1]}\nu\lrcorner\omega,\nu\lrcorner\omega\rangle+\langle {\rm{II}}^{[n-p]}\iota^*(*\omega),\iota^*(*\omega)\rangle=(n-1)H|\nu\lrcorner\omega|^2,$$
where we also use that $\iota^*(*\omega)=\pm *_{\partial M}(\nu\lrcorner\omega)$. Hence, after taking the real part, the above equation reduces to
\begin{eqnarray}\label{eq:intlapla}
{\rm{Re}}\left(\int_M\langle\Delta^\alpha\omega,\omega\rangle d\mu_g\right) =\int_M(|d^\alpha\omega|^2+|\delta^\alpha\omega|^2) d\mu_g-2 {\rm{Re}}\left( \int_{\partial M} \langle d^{\alpha^T}(\nu\lrcorner\omega),\iota^*\omega\rangle d\mu_g\nonumber \right)\\+{\rm Re}\left(\int_{\partial M} \langle\nabla^\alpha_\nu\omega,\omega\rangle d\mu_g\right) -\int_{\partial M}\langle{\rm{II}}^{[n-p]}\iota^*(*\omega),\iota^*(*\omega)\rangle d\mu_g\nonumber\\-\int_{\partial M}\langle{\rm{II}}^{[p]}\iota^*\omega,\iota^*\omega\rangle d\mu_g. \nonumber\\
\end{eqnarray}
Now, taking the Hermitian product of \eqref{eq:bochnermagnetic} with $\omega$, integrating over $M$ and taking the real part yields
\begin{eqnarray}\label{eq:nablastar}
{\rm{Re}}\left(\int_M\langle\Delta^\alpha\omega,\omega\rangle  d\mu_g\right) ={\rm{Re}}\left(\int_M\langle(\nabla^\alpha)^*\nabla^\alpha\omega,\omega\rangle d\mu_g\right) +\int_M\langle\mathcal{B}^{[p],\alpha}\omega,\omega\rangle d\mu_g\nonumber\\=\frac{1}{2}\int_M\Delta^M(|\omega|^2)d\mu_g+\int_M|\nabla^\alpha\omega|^2d\mu_g+\int_M\langle\mathcal{B}^{[p],\alpha}\omega,\omega\rangle d\mu_g\nonumber\nonumber\\
=\frac{1}{2}\int_{\partial M}\frac{\partial}{\partial\nu}(|\omega|^2)d\mu_g+\int_M|\nabla^\alpha\omega|^2d\mu_g+\int_M\langle\mathcal{B}^{[p],\alpha}\omega,\omega\rangle d\mu_g\nonumber\nonumber\\
={\rm{Re}}\left(\int_{\partial M}\langle\nabla^\alpha_\nu\omega,\omega\rangle d\mu_g\right) +\int_M|\nabla^\alpha\omega|^2d\mu_g+\int_M\langle\mathcal{B}^{[p],\alpha}\omega,\omega\rangle d\mu_g\nonumber.\\
\end{eqnarray}
The second equality is obtained by taking the real part of the pointwise identity $\langle(\nabla^\alpha)^*\nabla^\alpha\omega,\omega\rangle=-\sum_{i=1}^ne_i(\langle\nabla^\alpha_{e_i}\omega,\omega\rangle)+|\nabla^\alpha\omega|^2$ valid at any point such that $\nabla^M e_i=0$ and then using that ${\rm Re}(\langle\nabla^\alpha_X\omega,\omega\rangle)=\frac{1}{2}X(|\omega|^2)$ for any real vector field $X$. Comparing Equation \eqref{eq:intlapla} with Equation \eqref{eq:nablastar} yields the desired magnetic Reilly formula.
\end{proof}

Note that when $p=1$, by taking $\omega=d^\alpha f$ for any smooth complex-valued function $f$ and using the fact that $\mathcal{B}^{[1],\alpha}={\rm Ric}^M+iA^\alpha$ (here $A^{[1],\alpha}=-A^\alpha$ since $A^\alpha$ is skew-symmetric), the Reilly formula in Theorem \ref{thm:reilly} reduces to the one stated in \cite[Cor. 4.2]{ELMP:16} for manifolds without boundary and to \cite[Thm. 1.2]{HK:18} for manifolds with boundary.

\section{Eigenvalue estimates for the magnetic Hodge Laplacian on manifolds with boundary}

\subsection{A magnetic Raulot-Savo estimate}

In the following, we will estimate the first eigenvalue of the magnetic Laplacian on the boundary of an oriented Riemannian manifold in terms of the so-called $p$-curvatures as in \cite[Thm. 1]{RS:11}. We mainly follow and refer to \cite{RS:11} for further details. We consider a Riemannnian manifold $(M^n,g)$ with smooth boundary $\partial M$, and denote by $\eta_1(x)\leq\ldots\leq\eta_{n-1}(x)$ the eigenvalues of the Weingarten tensor ${\rm{II}} = -\nabla^M \nu$ at any point $x\in \partial M$. Here, as before, $\nu$ is the inward unit normal vector field to the boundary. For any $p\in\{1,\ldots,n-1\}$, the $p$-curvatures $\sigma_p(x)$ are defined as $\sigma_p(x):=\eta_1(x)+\ldots\eta_p(x)$ and we set $$\sigma_p(\partial M)=\mathop{\rm inf}\limits_{x\in \partial M}(\sigma_p(x)).$$
From Inequality \eqref{eq:upperbounda}, we have the following estimates
\begin{equation}\label{eq:inequweingarten}
\langle {\rm II}^{[p]}\omega,\omega\rangle\geq \sigma_p(\partial M) |\omega|^2 \quad\text{and}\quad \langle {\rm II}^{[p]}\omega,\omega\rangle\leq (\sigma_{n-1}(\partial M)-\sigma_{n-1-p}(\partial M)) |\omega|^2,
\end{equation}
for any $\omega\in \Omega^p(\partial M)$. Recall here that ${\rm II}^{[p]}$ is the canonical extension of $ {\rm II}$ to differential $p$-forms as in Equation \eqref{eq:extension}. Also, it is not difficult to check the following inequality $\frac{\sigma_p(x)}{p}\leq \frac{\sigma_q(x)}{q}$, for $p\leq q$, at any point $x$ on the boundary with equality if and only if $\eta_1(x)=\eta_2(x)=\ldots=\eta_q(x)$.

On manifolds with boundary, there are two notions of cohomology groups. We briefly recall them: The absolute cohomology group $H_A^p(M)$ which is defined  as the set of harmonic forms on $M$ satisfying the absolute boundary conditions, that is for any $p\in\{1,\ldots,n\}$,
$$H_A^p(M):=\{\omega\in \Omega^p(M,\mathbb{C})|\,\, d^M\omega=\delta^M\omega=0\,\, \text{on}\,\, M\,\, \text{and}\,\, \nu\lrcorner\omega=0\,\, \text{on}\,\, \partial M\}.$$
By Poincar\'e duality, the absolute cohomology group $H_A^p(M)$ is isomorphic to the relative cohomology group $H_R^{n-p}(M)$ which is defined as
$$H_R^p(M):=\{\omega\in \Omega^p(M,\mathbb{C})|\,\, d^M\omega=\delta^M\omega=0\,\, \text{on}\,\, M\,\, \text{and}\,\, \iota^*\omega=0\,\, \text{on}\,\, \partial M\}.$$
In \cite[Thm. 4]{RS:11}, the authors provide geometric obstructions to the vanishing of these cohomologies using the Reilly formula. Namely, these conditions are related to the Bochner operator on $M$ and to the $p$-curvatures of the boundary. Following the same idea, we will use the magnetic Reilly formula to deduce a similar vanishing result on the absolute cohomology groups by requiring a condition on the magnetic Bochner operator $\mathcal{B}^{[p],\alpha}$. We have the following result.
\begin{proposition}
Let $(M^n,g)$ be a compact Riemannian manifold with smooth boundary and let $\alpha$ be a differential $1$-form on $M$. Assume that $\mathcal{B}^{[p],\alpha}\geq |\alpha|^2$ and that $\sigma_p(\partial M)>0$. Then, $H_A^p(M)=0$.
\end{proposition}

\begin{proof} Let $\omega\in \Omega^p(M,\mathbb{C})$ be an element in $H_A^p(M)$. Applying the magnetic Reilly formula to $\omega$ and using the fact that $|d^\alpha\omega|^2+|\delta^\alpha\omega|^2=|\alpha|^2|\omega|^2$  yields the following:
$$
\int_M |\alpha|^2|\omega|^2 d\mu_g=\int_M |\nabla^\alpha \omega|^2 d\mu_g+ \int_M\langle \mathcal{B}^{[p],\alpha}\omega,\omega \rangle d\mu_g + \int_{\partial M} \langle \mathrm{II}^{[p]} \iota^* \omega,\iota^* \omega \rangle d\mu_g.$$
Now, the fact that $|\nabla^\alpha\omega|^2\geq 0$, the condition on $\mathcal{B}^{[p],\alpha}$ and Inequality \eqref{eq:inequweingarten} allow us to deduce that
\begin{eqnarray*}
\int_M |\alpha|^2|\omega|^2 d\mu_g&\geq&\int_M|\alpha|^2|\omega|^2 d\mu_g + \sigma_p(\partial  M)\int_{\partial M} |\iota^*\omega|^2  d\mu_g\\
&=&\int_M|\alpha|^2|\omega|^2 d\mu_g + \sigma_p(\partial  M)\int_{\partial M} |\omega|^2  d\mu_g\\
&\geq&\int_M|\alpha|^2|\omega|^2 d\mu_g.
\end{eqnarray*}
In the last inequality, we used that $\sigma_p(\partial  M)>0$. Hence, we have equality in the above inequalities and, thus, $\omega=0$ on $\partial M$. Now, since $\omega$ is harmonic, this leads to $\omega=0$ on $M$ by \cite{Ann:89}.
\end{proof}

In the following, we will consider a magnetic $1$-form $\alpha$ on $M$ such that its tangential part $\alpha^T=\iota^*\alpha$ is Killing of constant norm on $\partial M$. In this case, the exterior differential $d^{\partial M}$ and codifferential $\delta^{\partial M}$ commute with $\Delta^{\alpha^T}$ as we have seen in Proposition \ref{prop:deltaalphad}. Hence, as in \cite[Thm. 5]{RS:11}, we will estimate the first eigenvalue $\lambda_{1,p}^{\alpha^T}(\partial M)'$ of the magnetic Laplacian $\Delta^{\alpha^T}$ restricted to exact forms in terms of the $p$-curvatures.

\begin{theorem} \label{thm:rsestimate}
Let $(M^n,g)$ be a compact Riemannian manifold with smooth boundary $\partial M$ and let $\alpha$ be a differential $1$-form on $M$ such that $\alpha^T$ is a Killing form on $\partial M$ of constant norm. Assume that $\mathcal{B}^{[p],\alpha}\geq |\alpha|^2$  and that the $p$-curvatures $\sigma_p(\partial M)>0$ for some $1\leq p\leq \frac{n}{2}$. Then the first eigenvalue $\lambda_{1,p}^{\alpha^T}(\partial M)'$ satisfies the inequality
$$\lambda_{1,p}^{\alpha^T}(\partial M)'\geq \sigma_p(\partial M)\sigma_{n-p}(\partial M).$$
\end{theorem}
\begin{proof} Let $\omega=d^{\partial M}\beta$ be a complex exact $p$-eigenform of $\Delta^{\alpha^T}$ associated to the eigenvalue $\lambda_{1,p}^{\alpha^T}(\partial M)'$. From \cite[Lem. 3.1]{BS:08} (see also \cite[Lem. 3.4.7]{Sc:95}), there exists a complex $(p-1)$-form $\hat\beta$ such that
$\delta^M d^M\hat\beta=0,\, \delta^M\hat\beta=0$ on $M$ and $\iota^*\hat\beta=\beta$ on $\partial M$. The form $\hat\beta$ is unique up to a Dirichlet harmonic form, that is an element in $H^{p-1}_R(M)$. Notice here that $\hat\beta$ cannot be a Dirichlet harmonic form since this would lead to $\omega=0$. Let the $p$-form $\hat\omega:=d^M\hat\beta$ on $M$. Clearly, the form $\hat\omega$ satisfies the following system:
\begin{equation*}
\left\{
\begin{matrix}
	d^M\hat\omega=\delta^M\hat\omega=0& \text{on $M$,}\\
	\iota^*\hat\omega=\omega, & \text{on $\partial M.$}
\end{matrix}\right.
\end{equation*}
Applying the magnetic Reilly formula in Theorem \ref{thm:reilly} to the form $\hat\omega$ gives (after using that $|d^\alpha\hat\omega|^2+|\delta^\alpha\hat\omega|^2=|\alpha|^2|\hat\omega|^2$, the condition on the magnetic Bochner operator $\mathcal{B}^{[p],\alpha}$ and the fact that $|\nabla^\alpha\hat\omega|^2\geq 0$) the following inequality
\begin{equation}\label{eq:reillyestimate}
0\geq 2   {\rm{Re}}\left(  \int_{\partial M} \langle \nu \lrcorner \hat\omega,\delta^{\alpha^T} \omega \rangle d\mu_g\right) +\sigma_{p}(\partial M)\int_{\partial M}|\omega|^2d\mu_g+\sigma_{n-p}(\partial M)\int_{\partial M}|\nu\lrcorner\hat\omega|^2d\mu_g.
\end{equation}
We also use  the first estimate in \eqref{eq:inequweingarten} applied to the $p$-form $\iota^*\hat\omega=\omega$ and to the $(n-p)$-form  $\iota^*(*\hat\omega)=*_{\partial M}(\nu\lrcorner\hat\omega)$. As $p\leq \frac{n}{2}$, we have that $\frac{\sigma_p(\partial M)}{p}\leq \frac{\sigma_{n-p}(\partial M)}{n-p}$ and thus $\sigma_{n-p}(\partial M)>0$. Then, by using the pointwise inequality $|\nu\lrcorner\hat\omega+\frac{1}{\sigma_{n-p}(\partial M)}\delta^{\alpha^T}\omega|^2\geq 0$, we get the following estimate
$$\frac{2}{\sigma_{n-p}(\partial M)}{\rm{Re}}\left( \langle \nu \lrcorner \hat\omega, \delta^{\alpha^T}\omega \rangle \right)+|\nu\lrcorner\hat\omega|^2\geq -\frac{1}{\sigma_{n-p}(\partial M)^2}|\delta^{\alpha^T}\omega|^2.$$
Therefore by integrating this last inequality and multiplying it by $\sigma_{n-p}(\partial M)$, Inequality \eqref{eq:reillyestimate} reduces to
$$\frac{1}{\sigma_{n-p}(\partial M)}\int_{\partial M} |\delta^{\alpha^T}\omega|^2 d\mu_g\geq \sigma_{p}(\partial M)\int_{\partial M}|\omega|^2d\mu_g.$$
Finally, by using the fact that $\omega$ is a closed eigenform for the magnetic Laplacian $\Delta^{\alpha^T}$, we have
\begin{eqnarray*}
\lambda_{1,p}^{\alpha^T}(\partial M)'\int_{\partial M}|\omega|^2 d\mu_g &=&\int_{\partial M}(|d^{\alpha^T}\omega|^2+|\delta^{\alpha^T}\omega|^2)d\mu_g\\
&=&\int_{\partial M}(|\alpha^T\wedge\omega|^2+|\delta^{\alpha^T}\omega|^2)d\mu_g\\
&\geq& \int_{\partial M}|\delta^{\alpha^T}\omega|^2d\mu_g\\&\geq& \sigma_{p}(\partial M)\sigma_{n-p}(\partial M)\int_{\partial M}|\omega|^2d\mu_g.
\end{eqnarray*}
which is the desired estimate. This finishes the proof of the theorem.
\end{proof}

\subsection{A gap estimate between first eigenvalues}

In the next result, we adapt the computations in \cite[Thm. 2.3]{GS} to find a gap estimate between the eigenvalues of different degrees $\lambda_{1,p}^\alpha(M)$ and $\lambda_{1,p-1}^\alpha(M)$. For this, we will assume the manifold $(M^n,g)$ is isometrically immersed into Euclidean space $\mathbb{R}^{n+m}$  and consider the magnetic Laplacian with Dirichlet boundary conditions, in contrast to \cite{GS} where absolute boundary conditions are taken. Recall that for a given normal vector field $Z$ to $M$, the Weingarten tensor ${\rm II}_Z$ is the endomorphism of $TM$ given by
 $$\langle {\rm II}_Z(X),Y\rangle =\langle Z, {\rm II}(X,Y)\rangle$$
where $X,Y$ are tangent to $M$ and ${\rm II}$ is the second fundamental form of the immersion. As in Equation \eqref{eq:extension}, we will use the extension ${\rm II}_Z^{[p]}$ of the Weingarten tensor to $p$-differential forms.

 \begin{theorem}\label{gapestimatethm}
 Let $(M^n,g)$ be a compact manifold with smooth boundary that is isometrically immersed into the Euclidean space $\mathbb{R}^{n+m}$. Let $\alpha$ be a smooth $1$-form on $M$. Then, for all $1\leq p\leq n$, the eigenvalues of the magnetic Dirichlet Laplacian on $M$ satisfy
 $$ \lambda^\alpha_{1,p}(M) \geq \lambda^\alpha_{1,p-1}(M)+ \frac{1}{p} \sup_{x \in M}\lambda_{\rm{min}} \left( \mathcal{B}^{[p],\alpha}(x) -\sum_{t=1}^{m}({\rm II}_{f_t}^{[p]})^2(x)\right),$$
where $\lambda_{\rm{min}}(A)$ is the smallest eigenvalue of a symmetric operator $A$ and $\{f_1,\ldots,f_m\}$ is a local orthonormal basis of $TM^\perp$.
 \end{theorem}

\begin{proof} The proof follows along the lines of \cite{GS}. For each $j=1,\ldots,n+m$, the unit parallel vector field $\partial_{x_j}$ on $\mathbb{R}^{n+m}$ splits as $\partial_{x_j}=(\partial_{x_j})^T+(\partial_{x_j})^\perp$ with $(\partial_{x_j})^T=d^M(x_j\circ\iota)$ where $\iota$ is the isometric immersion. For any $p$-eigenform $\omega$ of $\Delta^{\alpha}$ associated to $\lambda_{1,p}^\alpha(M)$ with Dirichlet boundary condition, the $(p-1)$-form $(\partial_{x_j})^T\lrcorner\omega$  clearly satisfies the Dirichlet boundary condition. Hence, by the  characterization \eqref{eq:minmax} of the first eigenvalue applied to $(\partial_{x_j})^T\lrcorner\omega$, we have for each $j$,
\begin{equation}\label{eq:inequdiffdegre}
\lambda_{1,p-1}^\alpha(M)\int_M|(\partial_{x_j})^T\lrcorner\omega|^2d\mu_g\leq \int_M (|d^\alpha((\partial_{x_j})^T\lrcorner\omega)|^2+|\delta^\alpha((\partial_{x_j})^T\lrcorner\omega)|^2) d\mu_g.
\end{equation}
In the following, we will take the sum over $j$ and compute each term separately. For this, we let $\{e_1,\ldots,e_n\}$ denote a local orthonormal frame of $TM$. Recall that any complex $p$-form $\beta$ on $M$ can be written as $\beta=\frac{1}{p}\sum_{s=1}^n e_s^*\wedge (e_s\lrcorner\beta)$, and therefore, $\sum_{s=1}^n\langle e_s\lrcorner\beta,e_s\lrcorner\gamma\rangle=p\langle\beta,\gamma\rangle$ for any complex $p$-forms $\beta,\gamma$. Now, the sum over $j$ of the l.h.s. of \eqref{eq:inequdiffdegre} is equal to

\begin{eqnarray}\label{eq:normomega}
\sum_{j=1}^{n+m}|(\partial_{x_j})^T\lrcorner\omega|^2&=&\sum_{j=1}^{n+m}\sum_{s,t=1}^ng((\partial_{x_j})^T,e_s)g((\partial_{x_j})^T,e_t)\langle e_s\lrcorner\omega,e_t\lrcorner\omega\rangle\nonumber\\
&=&\sum_{s,t=1}^n\underbrace{\sum_{j=1}^{n+m}g(\partial_{x_j},e_s)g(\partial_{x_j},e_t)}_{\delta_{st}}\langle e_s\lrcorner\omega,e_l\lrcorner\omega\rangle\nonumber\\
&=&\sum_{s=1}^{n}|e_s\lrcorner\omega|^2=p|\omega|^2.
\end{eqnarray}
Now, using that $(\partial_{x_j})^T=d^M(x_j\circ\iota)$, we have that $\nabla^M(\partial_{x_j})^T={\rm Hess}^M(x_j\circ\iota)$, which is then a symmetic endomorphism on $TM$. Hence, it follows that  $$\delta^M((\partial_{x_j})^T\lrcorner\omega)=-\sum_{i=1}^n e_i\lrcorner \left(\nabla^M_{e_i}(\partial_{x_j})^T\lrcorner\omega\right)-(\partial_{x_j})^T\lrcorner\delta^M\omega=-(\partial_{x_j})^T\lrcorner\delta^M\omega.$$
In the last equality, we use the fact that $\sum_{i=1}^n e_i\lrcorner (A(e_i)\lrcorner)=0$ for any symmetric endomorphism $A$ of $TM$.  Therefore, we compute
\begin{eqnarray*}
\delta^\alpha((\partial_{x_j})^T\lrcorner\omega)&=&\delta^M((\partial_{x_j})^T\lrcorner\omega)-i\alpha\lrcorner( (\partial_{x_j})^T\lrcorner\omega)\\
&=&-(\partial_{x_j})^T\lrcorner\delta^M\omega+i(\partial_{x_j})^T\lrcorner(\alpha\lrcorner\omega)\\
&=&-(\partial_{x_j})^T\lrcorner\delta^\alpha\omega.
\end{eqnarray*}
Hence, we deduce that
\begin{equation}\label{eq:deltalpha}
\sum_{j=1}^{n+m}|\delta^\alpha((\partial_{x_j})^T\lrcorner\omega)|^2=\sum_{j=1}^{n+m}|(\partial_{x_j})^T\lrcorner\delta^\alpha\omega|^2=(p-1)|\delta^\alpha\omega|^2.
\end{equation}
In the last equality, we apply \eqref{eq:normomega} for $\delta^\alpha \omega$ instead of $\omega$. Now using Cartan's formula  
and the identity $\mathcal{L}_{X^T}\omega=\nabla^M_{X^T}\omega+{\rm II}_{X^\perp}^{[p]}\omega$ for any parallel vector field $X\in \mathbb{R}^{n+m}$ proven in \cite[formula (4.3)]{GS}, where ${\rm II}_{X^\perp}^{[p]}$ is defined in \eqref{eq:extension}, we write
\begin{eqnarray} \label{eq:dalpha}
d^\alpha((\partial_{x_j})^T\lrcorner\omega))&=&d^M((\partial_{x_j})^T\lrcorner\omega)+i\alpha\wedge \left((\partial_{x_j})^T\lrcorner\omega\right)\nonumber\\
&=&\mathcal{L}_{(\partial_{x_j})^T}\omega-(\partial_{x_j})^T\lrcorner d^M\omega+i\alpha\wedge \left((\partial_{x_j})^T\lrcorner\omega\right)\nonumber\\
&=&\nabla^M_{(\partial_{x_j})^T}\omega+{\rm II}_{(\partial_{x_j})^\perp}^{[p]}\omega-(\partial_{x_j})^T\lrcorner d^M\omega+i\alpha\wedge \left((\partial_{x_j})^T\lrcorner\omega\right)\nonumber\\
&=&\nabla^M_{(\partial_{x_j})^T}\omega+{\rm II}_{(\partial_{x_j})^\perp}^{[p]}\omega-(\partial_{x_j})^T\lrcorner d^\alpha\omega+i(\partial_{x_j})^T\lrcorner(\alpha\wedge\omega)\nonumber\\&&+i\alpha\wedge \left((\partial_{x_j})^T\lrcorner\omega\right)\nonumber\\
&=&\nabla^\alpha_{(\partial_{x_j})^T}\omega+{\rm II}_{(\partial_{x_j})^\perp}^{[p]}\omega-(\partial_{x_j})^T\lrcorner d^\alpha\omega.
\end{eqnarray}
In the last equality, we use the relation  $X\lrcorner(\alpha\wedge\omega)=\alpha(X)\omega-\alpha\wedge (X\lrcorner\omega)$ for any vector field $X$ and the definition of the magnetic covariant derivative $\nabla_X^\alpha=\nabla^M_X+i\alpha(X)$.  Now, we want to take the norm in  \eqref{eq:dalpha} and sum over $j$. We have
\begin{eqnarray*}
\sum_{j=1}^{n+m}|\nabla^\alpha_{(\partial_{x_j})^T}\omega|^2&=&\sum_{j=1}^{n+m}\sum_{s,t=1}^ng((\partial_{x_j})^T,e_s)g((\partial_{x_j})^T,e_t)\langle\nabla^\alpha_{e_s}\omega,\nabla^\alpha_{e_t}\omega\rangle\\
&=&\sum_{s,t=1}^n\underbrace{\sum_{j=1}^{n+m}g(\partial_{x_j},e_s)g(\partial_{x_j},e_t)}_{\delta_{st}}\langle\nabla^\alpha_{e_s}\omega,\nabla^\alpha_{e_t}\omega\rangle\\
&=&\sum_{s=1}^n|\nabla^\alpha_{e_s}\omega|^2=|\nabla^\alpha\omega|^2.
\end{eqnarray*}
We can do the same procedure for the cross terms in \eqref{eq:dalpha}, for example,
if we denote by $\{f_1,\ldots, f_m\}$ a local orthonormal frame of $TM^\perp$, we compute
\begin{eqnarray*}
\sum_{j=1}^{n+m}\langle \nabla^\alpha_{(\partial_{x_j})^T}\omega, {\rm II}_{(\partial_{x_j})^\perp}^{[p]}\omega  \rangle&=&\sum_{j=1}^{n+m}\sum_{s=1}^n \sum_{t=1}^m \langle (\partial_{x_j})^T,e_s \rangle \langle (\partial_{x_j})^\perp,f_t \rangle\langle\nabla^\alpha_{e_s}\omega, {\rm II}_{f_t}^{[p]}\omega \rangle
\\
&=&\sum_{s=1}^n\sum_{t=1}^m\underbrace{\sum_{j=1}^{n+m}g(\partial_{x_j},e_s)\langle \partial_{x_j},f_t\rangle}_{\langle e_s,f_t \rangle = 0}\langle\nabla^\alpha_{e_s}\omega,{\rm II}_{f_t}^{[p]}\omega
\rangle = 0.
\end{eqnarray*}
Therefore, all the terms involving $(\partial_{x_j})^T$ and $(\partial_{x_j})^\perp$ at the same time will vanish, and we get
\begin{eqnarray}\label{eq:dalphaomega}
\sum_{j=1}^{n+m}|d^\alpha((\partial_{x_j})^T\lrcorner\omega))|^2&=&|\nabla^\alpha\omega|^2+\sum_{t=1}^{m}|{\rm II}_{f_t}^{[p]}\omega|^2+(p+1)|d^\alpha\omega|^2\nonumber\\&&-2\sum_{s=1}^{n}{\rm{Re}}\left(\langle \nabla^\alpha_{e_s}\omega,e_s\lrcorner d^\alpha\omega\rangle\right)\nonumber\\
&=&|\nabla^\alpha\omega|^2+\sum_{t=1}^{m}|{\rm II}_{f_t}^{[p]}\omega|^2+(p-1)|d^\alpha\omega|^2.\nonumber\\
\end{eqnarray}
Replacing \eqref{eq:normomega}, \eqref{eq:deltalpha} and \eqref{eq:dalphaomega} into Inequality \eqref{eq:inequdiffdegre}, we obtain $$\lambda^\alpha_{1,p-1}(M) p\int_M|\omega|^2 d\mu_g\leq \int_M\left(|\nabla^\alpha\omega|^2+\sum_{t=1}^{m}|{\rm II}_{f_t}^{[p]}\omega|^2+(p-1)(|d^\alpha\omega|^2+|\delta^\alpha\omega|^2)\right)d\mu_g.$$
Now using Equality \eqref{eq:nablastar} for the eigenform $\omega$ with Dirichlet boundary conditions yields that $$\int_M|\nabla^\alpha\omega|^2d\mu_g=\lambda^\alpha_{1,p}(M) \int_M|\omega|^2 d\mu_g-\int_M\langle\mathcal{B}^{[p],\alpha}\omega,\omega\rangle d\mu_g.$$
Hence, we deduce that
\begin{eqnarray*}
\lambda^\alpha_{1,p-1}(M) p\int_M|\omega|^2 d\mu_g&\leq& p\lambda^\alpha_{1,p}(M)\int_M|\omega|^2 d\mu_g -\int_M\langle\mathcal{B}^{[p],\alpha}\omega,\omega\rangle d\mu_g\\ &&+\sum_{t=1}^{m}\int_M\langle({\rm II}_{f_t}^{[p]})^2\omega,\omega\rangle d\mu_g,
\end{eqnarray*}
which ends the proof.
\end{proof}

\begin{corollary}
Let $(M^n,g)$ be a domain in Euclidean space $\mathbb{R}^n$ and let $\alpha$ be a $1$-form on $M$. Then, for all $p\geq 1$, the eigenvalues of the magnetic Dirichlet Laplacian satisfy
$$\lambda^{\alpha}_{1,p}(M)\geq \lambda^\alpha_{1,p-1}(M)-||d^M\alpha||_\infty.$$
In particular, the following estimate
$$\lambda_{1,p}^\alpha(M)\geq \lambda_0(M)-p||d^M\alpha||_\infty$$
holds, where $\lambda_0(M)$ is the first eigenvalue of the scalar Laplacian with Dirichlet boundary condition.
\end{corollary}
\begin{proof}
Since $M$ is a domain in Euclidean space, the second fundamental form and the curvature operator of $M$ vanish. Therefore, Theorem \ref{gapestimatethm} allows us to deduce that
$$\lambda^{\alpha}_{1,p}\geq \lambda^\alpha_{1,p-1} +
\frac{1}{p} \sup_{x \in M} \lambda_{\rm{min}} \left(- i A^{[p],\alpha}\right) .$$
Recall here that $-iA^{[p],\alpha}$ is a symmetric tensor field where $A^\alpha(X)=X\lrcorner d^M\alpha$ for all $X\in TM$. Now, by the second inequality in \eqref{eq:upperbounda}, we have $iA^{[p],\alpha}\leq p||A^\alpha||\leq p||d^M\alpha||_\infty$. This finishes the first part. The second part is easily proved by taking
successive $p$'s.
\end{proof}

\begin{corollary}
Let $(M^n,g)$ be a domain in the round unit sphere $\mathbb{S}^n$ and let $\alpha$ be a $1$-form on $M$. Then, for all $p\geq 1$, the eigenvalues of the magnetic Dirichlet Laplacian satisfy
$$\lambda^{\alpha}_{1,p}(M)\geq \lambda^\alpha_{1,p-1}(M)+n-2p-||d^M\alpha||_\infty.$$
In particular, the following estimate
$$\lambda_{1,p}^\alpha(M)\geq \lambda_0(M)+p(n-p-1-||d^M\alpha||_\infty)$$
holds, where $\lambda_0(M)$ is the first eigenvalue of the scalar Laplacian with Dirichlet boundary condition.
\end{corollary}

\begin{proof}
We use the isometric immersion of $\mathbb{S}^n\hookrightarrow \mathbb{R}^{n+1}$ for which the second fundamental form is the identity. The proof is then a direct consequence of Theorem \ref{gapestimatethm} using the fact that, on the round sphere, $\mathcal{B}^{[p]}=p(n-p)$ and that
$\sum_{a=1}^{m}({\rm II}^{[p]}_{f_a})^2=p^2$.
\end{proof}

\appendix

\section{Spectral computations for magnetic Laplacians for functions on Berger spheres}\label{sec:berger}

\subsection{Eigenvalue decomposition of the ordinary Laplacian on the standard $3$-sphere}

The following considerations are based on the arguments given in \cite[pp. 27]{Hi74}. For further details see also \cite[III.3-III.7]{norbert-thesis}.

Let $\mathbb{S}^3=\{(z_1, z_2)\in \CC^2\, \vert \, \abs{z_1}^2+\abs{z_2}^2=1\}$ be the $3$-dimensional unit sphere and let $g$ be the standard metric on $\mathbb{S}^3$ of curvature one. We can also think of $\mathbb{S}^3$ as the Lie group of all unit quaternions via the identification $(z_1,z_2) \mapsto z_1 + j z_2 \in \HH^2$. Let $Y_2,Y_3,Y_4$ be the left-invariant extensions of the tangent vectors $i, -k, -j\in T_1 \mathbb{S}^3$. In this case, the vectors
\begin{eqnarray*}
Y_2 &=& -y_1 \partial_{x_1} + x_1 \partial_{y_1} - y_2 \partial_{x_2} + x_2 \partial_{y_2}, \\
Y_3 &=& -y_2 \partial_{x_1} - x_2 \partial_{y_1} + y_1 \partial_{x_2} + x_1 \partial_{y_2}, \\
Y_4 &=& x_2 \partial_{x_1} - y_2 \partial_{y_1} - x_1 \partial_{x_2} + y_1 \partial_{y_2}
\end{eqnarray*}
form an orthonormal basis of $T_{(z_1,z_2)}\mathbb{S}^3$ at every point $(z_1,z_2) = (x_1 + y_1 i, x_2 + y_2 i) \in \mathbb{S}^3$.

Then, we can write the Laplacian on $(\mathbb{S}^3, g)$ as $\Delta^{\mathbb{S}^3} f= -\sum_{j=2}^4 Y_j^2(f)$ for all $f\in C^\infty(\mathbb{S}^3)$,
whose eigenvalues are $\lambda_k(\mathbb{S}^3)=k(k+2)$, $k\in\NN\cup \{0\}$ with multiplicity $(k+1)^2$. In particular, every eigenspace $E_{k}$ associated with the eigenvalue $\lambda_k$ decomposes as
\begin{equation} \label{eq:Ekdecomp}
E_k= V_{k, (a_0,b_0)} \oplus V_{k, (a_1, b_1)}\oplus \ldots \oplus V_{k, (a_k, b_k)},
\end{equation}
with any arbitrary choice of pairwise non-collinear vectors $(a_j,b_j) \in \CC \setminus\{(0,0)\}$, where
\begin{align*}
& V_{k, (a,b)}  = {\textrm{span}}_\CC\{u_{a,b}^k, u_{a,b}^{k-1}v_{a,b}, \ldots, u_{a,b} v_{a,b}^{k-1}, v_{a,b}^k\}, \\
& u_{a,b}(z_1, z_2):= az_1+bz_2, \quad v_{a,b}(z_1, z_2):=b \bar z_1 -a\bar z_2,
\end{align*}
for $(a,b)\in\CC^2\setminus\{(0,0)\}$, see \cite[Zerlegungssatz III.6.2]{norbert-thesis}. For short, we write $u:= u_{(a,b)}$, $v:= v_{(a,b)}$ for some $(a,b)\neq (0,0)$ and, for $p \in \{0,\ldots, k\}$, we consider
$$
\phi_p:= u^p v^{q-1}
$$
with $p+q=k+1$. (We also set $\phi_p \equiv 0$ for all other choices of $p$.) These functions $\phi_p$ are spherical harmonics, that is, they are restrictions of harmonic homogeneous polynomials on $\CC^2$ to the unit sphere $\mathbb{S}^3$. Then we have $V_{k, (a,b)}={\textrm{span}}_\CC\{\phi_0,\ldots, \phi_k\}$.
A straightforward computation yields (see \cite[p. 30]{Hi74} or \cite[Lemma III.7.1]{norbert-thesis})
\begin{eqnarray}
Y_2(\phi_p) &=& i(p-q+1)\phi_p, \label{eq:Y2phi} \\
Y_3(\phi_p) &=& ip\phi_{p-1}+i(q-1)\phi_{p+1}, \label{eq:Y2phi2}\\
Y_4(\phi_p) &=& -p\phi_{p-1}+(q-1)\phi_{p+1}, \label{eq:Y2phi3}\\
(Y_3^2+Y_4^2)(\phi_p) &=& 2(p-2pq-q+1)\phi_p. \nonumber
\end{eqnarray}
This implies
\[
\Delta^{\mathbb{S}^3}  \phi_p = - \sum_{j=2}^4 Y_j^2(\phi_p) = [ (p+q)^2-1 ] \phi_p = k(k+2) \phi_p,
\]
confirming that the functions $\phi_p$ are eigenfunctions of $\Delta^{\mathbb{S}^3}$ in the eigenspace $E_k$.

Let us briefly describe the underlying representation theory. The Lie group ${\rm SU}(2)$ acts irreducibly on each of the vector spaces $V_{k,(a,b)} \subset \CC[z_1,\bar z_1,z_2,\bar z_2]$ via
$$ \rho: {\rm SU}(2) \times V_{k,(a,b)} \to V_{k,(a,b)}, \quad
\rho(A,P(u,v)) = P( (u,v) \cdot A ), $$
where $P \in \CC[w_1,w_2]$ is any homogenous polynomial of degree $k$.
Using the decomposition \eqref{eq:Ekdecomp}, these irreducible representations $\rho_j$ on each of the factors $V_{k,(a_j,b_j)}$ give rise to the ${\rm SU}(2)$-representation
$$\mu_k := \rho_0 \oplus \rho_1 \oplus \ldots \oplus \rho_k$$
on the eigenspace $E_k$.

On the other hand, the identification of $\mathbb{S}^3$ with the Lie group ${\rm SU}(2)$ via
$$ (z_1,z_2) \mapsto \begin{pmatrix} z_1 & - \bar z_2 \\ z_2 & \bar z_1 \end{pmatrix} $$
provides a canonical isometric ${\rm SU}(2)$-right action on $(\mathbb{S}^3,g)$, which leads to the corresponding unitary ${\rm SU}(2)$-action $$ (Af)(z_1+jz_2) = f((z_1+jz_2)(\alpha+j\beta)) \quad \text{for}\, A =
\begin{pmatrix} \alpha & -\bar \beta \\ \beta & \bar \alpha \end{pmatrix} \in {\rm SU}(2)$$
on the function space $C^\infty(\mathbb{S}^3) \subset L^2(\mathbb{S}^3,g)$. Since $\Delta^{\mathbb{S}^3}$ commutes with isometries, the eigenspace $E_k \subset C^\infty(\mathbb{S}^3)$ is an invariant subspace of this latter action, and its restriction to $E_k$ agrees with the above ${\rm SU}(2)$-representation $\mu_k$ (see \cite[Lemma III.6.5]{norbert-thesis}).

\medskip

Now let $\mathbb{S}^1 \hookrightarrow \mathbb{S}^3 \rightarrow \mathbb{S}^2$ be the Hopf fibration of $(\mathbb{S}^3,g)$, where the fiber through a  point $(z_1,z_2) \in \mathbb{S}^3$ is given by $ F_{(z_1,z_2)} := \{ (e^{it} z_1,e^{it} z_2) \vert\, t \in \RR \} \subset \mathbb{S}^3$. The map $\mathbb{S}^3 \to \mathbb{S}^2$ is a Riemannian submersion, the  fibers are totally geodesic, and we have
\[
T_{(z_1, z_2)}\mathbb{S}^3 = V_{(z_1,z_2)} \oplus H_{(z_1, z_2)}
\]
for any $(z_1, z_2)\in \mathbb{S}^3$, where the vertical component $V_{(z_1,z_2)}$ is spanned by $Y_2$ and the horizontal component $H_{(z_1,z_2)}$ is spanned by
$Y_3$ and $Y_4$. This decomposition induces a corresponding splitting
$$ \Delta^{\mathbb{S}^3} = \Delta^v + \Delta^h $$
of $\Delta^{\mathbb{S}^3}$ into a vertical and a horizontal Laplacian
$\Delta^v$ and $\Delta^h$ (see \cite[Def. 1.2 and 1.3]{BBB:81}) with
\[
\Delta^v = - Y_2^2 \quad \text{and}\,\, \Delta^h =-(Y_3^2+ Y_4^2).
\]
Since the fibres are totally geodesic, the three operators $\Delta^{\mathbb{S}^3}, \Delta^v, \Delta^h$ commute with each other, and $L^2(\mathbb{S}^3)$ admits a Hilbert basis
consisting of simultaneous eigenfunctions of $\Delta^{\mathbb{S}^3}$ and $\Delta^h$ (see \cite{BBB:81}). In our case, this Hilbert basis is obtained through the eigenspaces $E_k$ and their decompositions into the subspaces $V_{k,(a,b)}$, whose corresponding basis vectors $\phi_p$, $p \in \{0,\dots,k\}$, are then the members of this Hilbert basis.

\subsection{Geometry of Berger spheres}

Given the standard metric $g$ on $\mathbb{S}^3$ of curvature $1$ and $\epsilon > 0$, the \emph{Berger sphere} is the Riemannian manifold
$(\mathbb{S}^3,g_\epsilon)$ with
$$ g_\epsilon = \epsilon^2 g\vert_{V\times V} \oplus
g\vert_{H\times H}, $$
and the vector fields $ Y^\epsilon_2:=\epsilon^{-1}Y_2, Y_3^\epsilon:= Y_3, Y_4^\epsilon:= Y_4$ form a global orthonormal frame. The Lie brackets are given by
\[
[Y_2^\epsilon, Y_3^\epsilon] =-\frac{2}{\epsilon}Y_4^\epsilon \qquad
[ Y_2^\epsilon, Y_4^\epsilon] = \frac{2}{\epsilon} Y_3^\epsilon\qquad
[Y_3^\epsilon, Y_4^\epsilon] =-2\epsilon  Y_2^\epsilon,
\]
and the Christoffel symbols of the Levi-Civita connection of $g_\epsilon$ are expressed as
\begin{equation} \label{eq:covYiYj}
\nabla_{Y_j^\epsilon}^{\mathbb{S}^3} Y_k^\epsilon = \sigma_{jk} Y_l^\epsilon
\end{equation}
with $\{j,k,l\} = \{2,3,4\}$ for $k \neq j$, $\sigma_{jj} = 0$ and $\sigma_{23}=-\sigma_{24} = \epsilon- 2/\epsilon$, $\sigma_{32}=\sigma_{43} = - \sigma_{34} = -\sigma_{42} = \epsilon$. In particular, we deduce that
\begin{equation}\label{eq:exteriory2}
 d^{\mathbb{S}^3} Y_2^\epsilon=2\epsilon Y_3^\epsilon\wedge Y_4^\epsilon,  d^{\mathbb{S}^3} Y_3^\epsilon=-\frac{2}{\epsilon} Y_2^\epsilon\wedge Y_4^\epsilon,\, d^{\mathbb{S}^3} Y_4^\epsilon=\frac{2}{\epsilon} Y_2^\epsilon\wedge Y_3^\epsilon  \quad\text{and}\quad \delta^{\mathbb{S}^3} Y_j^\epsilon=0.
 \end{equation}
for $j\in\{2,3,4\}$. Here $\delta^{\mathbb{S}^3}$ is the $L^2$-adjoint of $d^{\mathbb{S}^3}$ with respect to the metric $g_\epsilon$. The curvature tensor associated to the Levi-Civita connection of $g_\epsilon$ can be computed explicitly and is equal to
$$ R^{\mathbb{S}^3}(Y_j^\epsilon,Y_k^\epsilon)Y_l^\epsilon = \tau_{jkl} Y_j^\varepsilon $$
with $\tau_{jkl} = 0$ for $\{j,k,l\} = \{2,3,4\}$, $\tau_{233}=\tau_{244}=\tau_{322}=\tau_{422} = \epsilon^2$ and
$\tau_{344}=\tau_{433}=4-3\epsilon$.
The sectional curvatures of the planes spanned by pairs of $Y_i^\epsilon$'s  are
$$K^{\mathbb{S}^3}({\textrm{span}}\{ Y_2^\epsilon, Y_3^\epsilon\}) =
	K^{\mathbb{S}^3}({\textrm{span}}\{ Y_2^\epsilon, Y_4^\epsilon\}) =  \epsilon^2, \,
	K^{\mathbb{S}^3}({\textrm{span}}\{ Y_3^\epsilon, Y_4^\epsilon\}) = 4-3\epsilon^2.$$
The Ricci tensor of any vector $v= \sum_{j=2}^4 a_j Y_j^\epsilon$ is given by
\[
\Ric^{\mathbb{S}^3}(v,v)=\sum_{j=2}^4 g_\epsilon(R^{\mathbb{S}^3}(v, Y_j^\epsilon)Y_j^\epsilon,v)= 2 \epsilon^2 a_2^2 + (4-2\epsilon^2)(a_3^2+a_4^2),
\]
which yields the following lower Ricci curvature bounds
\begin{equation}\label{eq:ricci-bounds}
\inf_{\Vert v \Vert =1} \left(\Ric^{\mathbb{S}^3}(v,v)\right) \ge \left\{
\begin{matrix}
	2\epsilon^2, & \text{if $\epsilon \le 1$,}\\
	4-2\epsilon^2, & \text{if $\epsilon >1$.}
\end{matrix}\right.
\end{equation}
Observe, moreover, that $\lim_{\epsilon \to 0} \Ric^{\mathbb{S}^3}(v) =  4(a_3^2+a_4^2)$.
Since the ``scaled" Hopf fibration $\mathbb{S}^1_\epsilon \hookrightarrow \mathbb{S}^3 \rightarrow \mathbb{S}^2$ is a Riemannian submersion with totally geodesic fiber $\mathbb{S}^1_\epsilon$
(see \cite[Prop. 5.2]{BBB:81}), any horizontal vector $v^h\in H_{(z_1,z_2)}$ for $(z_1, z_2)\in \mathbb{S}^3$ is uniquely mapped to a vector
$\tilde v \in T\mathbb{S}^2$
and so we can say that, as $\epsilon \to 0$, the Ricci curvature of $(\mathbb{S}^3, g_\epsilon)$ collapses to the Ricci curvature of $\CP^1$ with the Fubini-Study metric.

\subsection{Eigenvalue decomposition of the ordinary Laplacian on Berger Spheres}

In this subsection, we will compute the eigenvalues of the Laplacian on the Berger sphere $\mathbb{S}^3$. We refer to \cite[Lem. 4.1]{tanno:79}, \cite[Prop. 3.9]{Lauret:19} for similar results.

Since $Y_2^\epsilon, Y_3^\epsilon, Y_4^\epsilon$ form a global divergence-free orthonormal frame by \eqref{eq:exteriory2}, the Laplacian on $(\mathbb{S}^3, g_\epsilon)$ is given by
\[
\Delta^{\mathbb{S}^3}_\epsilon f = -\sum_{j=2}^4 (Y_j^\epsilon)^2 f \qquad \forall \, f\in C^\infty(\mathbb{S}^3).
\]
Using the fact that $\mathbb{S}^1_\epsilon \hookrightarrow \mathbb{S}^3\rightarrow \mathbb{S}^2$ is a Riemannian submersion with totally geodesic fibers, we can write
\[
\Delta^{\mathbb{S}^3}_\epsilon = \Delta_\epsilon^v + \Delta_\epsilon^h = \epsilon^{-2} \Delta^v + \Delta^h = \Delta^{\mathbb{S}^3} + (\epsilon^{-2}-1)\Delta^v,
\]
where $\Delta^v, \Delta^h$ are the vertical and horizontal Laplacian w.r.t. $g$ and $\Delta^{\mathbb{S}^3}$ is the Laplacian on $\mathbb{S}^3$ w.r.t $g$.

Since $\{\phi_p\}_{p}$ is a Hilbert basis for $L^2(\mathbb{S}^3, g)$,
the set $\{\phi_p^\epsilon:= \epsilon^{1/2}\phi_p\}_p$ is a Hilbert basis for $L^2(\mathbb{S}^3, g_\epsilon)$.
Moreover, the functions $\phi_p^\epsilon$'s are eigenfunctions for $\Delta^{\mathbb{S}^3}_\epsilon$:
\begin{align*}
	\Delta^{\mathbb{S}^3}_\epsilon \phi_p^\epsilon & = k(k+2)  \phi_p^\epsilon + (\epsilon^{-2}-1) (p-q+1)^2 \phi_p^\epsilon \\
	& = \big[k(k+2) + \big(\frac{1}{\epsilon^2}-1\big)(2p-k)^2\big]\phi_p^\epsilon.
\end{align*}
The eigenvalues of $\Delta^{\mathbb{S}^3}_\epsilon$ are therefore all of the form
\[
 k(k+2) + \big(\frac{1}{\epsilon^2}-1\big)(2p-k)^2, \qquad k\in\NN\cup \{0\}, \quad p\in\{0,\ldots, k\}.
\]
One could also read off the spectrum of the vertical Laplacian $\Delta^v$ from \cite[Lemma 3.1]{tanno:79}.
The following table lists the eigenvalues for $k \in \{0,1,2,3\}$:\\

\begin{center}
\begin{tabular}{|c|c|c|}
	\hline
	$k$ &  $p$ &  $\lambda_{k,p}^\epsilon$\\
	\hline
	$0$ & $0$  & $0$ \\
	\hline
	$1$ & $0,1$ &  $2+\epsilon^{-2}$\\
	\hline
	$2$ & $ 0,2$  & $4+4\epsilon^{-2}$ \\
	\hline
	$2$ & $1$  & $8$ \\
	\hline
	$3$ & $0,3$  & $6+9\epsilon^{-2}$ \\
	\hline
	$3$ & $1,2$ & $14+\epsilon^{-2}$\\
	\hline
\end{tabular}
\end{center}

Therefore the first non-zero eigenvalue of $\Delta_\varepsilon^{\mathbb{S}^3}$ is $8$ if $\epsilon\leq 1/\sqrt 6$ and $2+\epsilon^{-2}$
if $\epsilon>1/\sqrt{6}$.
Moreover, all the eigenvalues of $\Delta_\varepsilon^{\mathbb{S}^3}$ tend to $\infty$, if $k\neq 2p$, when $\epsilon$ tends to $0$ and are equal to $k(k+2)$ if $k=2p$.
Hence, as $\epsilon\to 0$, the only eigenvalues not escaping to infinity are the ones coming from the Laplacian on $\CP^1$ with the Fubini-Study metric
(recall that its spectrum is $4p(p+1)=k(k+2)$ with $p\in\NN \cup \{0\}$ and $k=2p$).

\subsection{The magnetic Laplacian with constant magnetic potential along the fibers on Berger spheres}

As before let $(\mathbb{S}^3,g_\epsilon)$ be the Berger sphere and set $\alpha:=\varepsilon t Y_2^{\varepsilon}=tY_2$, by the identification through musical isomorphisms. Then $|\alpha|^2=\epsilon^2 t^2$ and $\delta^{\mathbb{S}^3}\alpha=0$ by \eqref{eq:exteriory2}. Therefore for the magnetic Laplacian  $\Delta_\epsilon^\alpha$ on $(\mathbb{S}^3, g_\epsilon)$ we have,
\[
\Delta_\epsilon^\alpha f = \Delta_\epsilon^{\mathbb{S}^3} f -2i\alpha^\sharp(f) + \epsilon^2 t^2 f.
\]
Applying this identity to the functions $f:=\phi_p^\epsilon=\epsilon^{1/2}\phi_p$ yields

$$\Delta_\epsilon^\alpha \phi_p^\epsilon =\left(k(k+2)+ \left(\frac{1}{\epsilon^2}-1\right)(2p-k)^2+2(2p-k)t+\epsilon^2 t^2\right)\phi_p^\epsilon,$$
since
\begin{equation} \label{eq:alphagradphi}
\alpha^\sharp(\phi_p^\epsilon) = t Y_2(\phi_p^\epsilon) = it(p-q+1)\phi_p^\epsilon=it(2p-k)\phi_p^\epsilon
\end{equation}
by \eqref{eq:Y2phi} and $p+q=k+1$.
Therefore the spectrum of $\Delta_\epsilon^\alpha$ is given by
\begin{equation} \label{eq:specmagS3}
k(k+2)+\left(\frac{1}{\epsilon^2}-1\right)(2p-k)^2+2(2p-k)t+\epsilon^2 t^2, \quad k\in\NN\cup\{0\},\,\, p\in\{0,\ldots, k\}.
\end{equation}
If $\epsilon\to 0$ (that is, if we are shrinking the fibers), the only eigenvalues not escaping to infinity, are $k(k+2)$ for even integers $k \ge 0$, that is, the eigenvalues of the Laplacian on $\CP^1$ with Fubini-Study metric. In other words, the magnetic potential disappears under this process.

\section{Special eigen-$1$-forms of the (magnetic) Hodge Laplacian on $\mathbb{S}^3$}
\label{sec:dudvalphaeig}

Let $\alpha=\varepsilon t Y_2^\varepsilon$ be a Killing vector field of constant norm, then by Proposition \ref{prop:deltaalphad} we have that $\Delta_\epsilon^\alpha(d^{\mathbb{S}^3} u)=d^{\mathbb{S}^3}(\Delta_\epsilon^\alpha u)$ and $\Delta_\epsilon^\alpha(d^{\mathbb{S}^3} v)=d^{\mathbb{S}^3}(\Delta_\epsilon^\alpha v)$. Now, by Equation \eqref{eq:specmagS3}, for the function $u$ (which corresponds to $p=q=k=1$) and the function $v$ (which corresponds to $p=0,\, q=2, k=1$), we compute
$$\Delta_\epsilon^\alpha u=(2+\frac{1}{\varepsilon^2}+2t+\epsilon^2 t^2)u\quad\text{and}\quad \Delta_\epsilon^\alpha v=(2+\frac{1}{\epsilon^2}-2t+\epsilon^2 t^2)v.$$
Hence $d^{\mathbb{S}^3}u$ and $d^{\mathbb{S}^3}v$ are eigenforms of $\Delta_\epsilon^\alpha$ corresponding to the eigenvalues $(2+\frac{1}{\epsilon^2}+2t+\epsilon^2 t^2)$ and $(2+\frac{1}{\epsilon^2}-2t+\epsilon^2 t^2)$ respectively.

To compute $\Delta_\epsilon^\alpha Y_2^\varepsilon$, we first have by \eqref{eq:exteriory2} that $Y_2^\varepsilon$ is coclosed and $d^{\mathbb{S}^3} Y_2^\epsilon=2\epsilon Y_3^\epsilon\wedge Y_4^\epsilon$. Thus, by \eqref{eq:covYiYj}, we get that  $\Delta_\epsilon^{\mathbb{S}^3} Y_2^\epsilon=4\epsilon^2 Y_2^\epsilon$. Also, we have
$$A^{[1],\alpha}Y_2^\epsilon=-A^\alpha(Y_2^\epsilon)=-Y_2^\epsilon\lrcorner (2\epsilon tY_3^\epsilon\wedge Y_4^\epsilon)=0, \quad\text{and}\quad \nabla^M_\alpha Y_2^\epsilon=0.$$
Therefore, by Equation \eqref{eq:deltaalphaforms},  we get that $\Delta_\epsilon^\alpha Y_2^\epsilon=\epsilon^2(4+t^2)Y_2^\epsilon.$ In the same way one can check that
$$\Delta_\epsilon^\alpha Y_3^\epsilon=(\frac{4}{\epsilon^2}+t^2\epsilon^2)Y_3^\epsilon+2i\epsilon t(1-\epsilon+\frac{2}{\epsilon})Y_4^\epsilon,$$
and that
$$\Delta_\epsilon^\alpha Y_4^\epsilon=(\frac{4}{\epsilon^2}+t^2\epsilon^2)Y_4^\epsilon-2i\epsilon t(1-\epsilon+\frac{2}{\epsilon})Y_3^\epsilon.$$
Hence for $\epsilon=2$, we get that $Y_3^\epsilon$ and $Y_4^\epsilon$ are eigenvectors associated to the eigenvalue $1+4t^2$.


\begin{thebibliography}{9}

\bibitem{Ann:89} C. Ann\'e,
	\emph{Principe de Dirichlet pour les formes diff\'erentielles}, Bull. Soc. Math. France \textbf{117}  (1989), 445--450.

\bibitem{BBC:03} W. Ballman, J. Br\"uning, G. Carron,
	\emph{Eigenvalues and holonomy}, Int. Math. Res. Not. \textbf{12}  (2003), 657--665.
	
\bibitem{BS:08} M.~Belishev, V. Sharafutdinov,
	\emph{Dirichlet to Neumann operator on differential forms},
	Bull. Sci. math. \textbf{132} (2008), 128--145.

\bibitem{BBB:81} L.~B\'erard Bergery, J.-P. Bourguignon,
	\emph{Laplacians and Riemannian submersions with totally geodesic fibres},
	Illinois J. of Math. \textbf{26} (2) (1982), 181--200.

\bibitem{BDP:16} V. Bonnaillie-No\"el, M. Dauge, N. Popoff,
	\emph{Ground state energy of the magnetic Laplacian on corner domains}, M\'em. Soc. Math. Fr.  \textbf{145}  (2016), vii+138.

\bibitem{CESIS-17} B. Colbois, A. El Soufi, S. Ilias, A. Savo, \emph{Eigenvalue upper bounds for the magnetic Schr\"odinger operator}, 	arXiv: 1709.09482.


\bibitem{CS:22} B. Colbois, L. Provenzano, A. Savo, \emph{Isoperimetric inequalities for the magnetic Neumann and Steklov problems with Aharonov-Bohm magnetic potential}, 
    J. Geom. Anal. \textbf{32} (2022), 285.

\bibitem{CS18} B. Colbois,  A. Savo, \emph{Lower bounds for the first eigenvalue of the magnetic {L}aplacian}, J. Funct. Anal. {\bf{274}} (10) (2018), 2818--2845.
 	
 	
\bibitem{CS21} B. Colbois,  A. Savo, \emph{Lower bounds for the first eigenvalue of the Laplacian with zero magnetic field in planar domains}, J. Funct. Anal. \textbf{281} (2021), 108999.

\bibitem{CS:21} B. Colbois,  A. Savo, \emph{Upper bounds for the ground state energy of the Laplacian with zero magnetic feld on planar domains}, Ann. Glob. Anal. Geom. \textbf{60} (2021), 1-18.

\bibitem{ELMP:16} M.~Egidi, S.~Liu, F.~M\"unch, N.~Peyerimhoff,
	\emph{Ricci curvature and eigenvalue estimates for the magnetic Laplacian on manifolds},
	Comm. Anal. Geom. {\bf{29}} (2021), 1127--1156.
	
\bibitem{L:96} L. Erd\"os, \emph{Rayleigh-type isoparametric inequality with a homogeneous magnetic field}, Bull. Lond. Math. Soc. \textbf{4} (1996), 283--292.

\bibitem{GM:75} S.~Gallot, D.~Meyer,
	\emph{Op\'erateur de courbure et laplacien des formes diff\'erentielles d'une vari\'et\'e riemannienne},
J. Math. Pures. Appl. \textbf{54} (1975), 259--284.	

\bibitem{GS} P.~Guerini, A.~Savo,
	\emph{Eigenvalue and gap estimates for the Laplacian on $p$-forms},
Trans. Amer. Math. Soc. \textbf{356} (2003), 319--344.	

\bibitem{HK:18} G.~Habib, A.~Kachmar,
     \emph{Eigenvalue bounds of the {R}obin {L}aplacian with magnetic field}, Arch. Math. \textbf{110} (2018), 501--513.

\bibitem{HOOO:99} B. Helffer, M. Hoffmann-Ostenhof,  T. Hoffmann-Ostenhof, M. P. Owen, \emph{Nodal sets for groundstates of Schr\"odinger operators with zero magnetic field in non-simply connected domains}, Comm. Math. Phys. \textbf{202} (1999), 629-649.

\bibitem{Hi74} N. Hitchin, \emph{Harmonic spinors}, Adv. Math. \textbf{14} (1974), 1--55.


\bibitem{Ka95} T. Kato, \emph{Perturbation theory for linear operators}, Springer, Berlin, (1995).

\bibitem{LLPP:15} C. Lange, S. Liu, N. Peyerimhoff, O. Post, \emph{Frustration index and Cheeger inequalities for discrete and continuous magnetic Laplacians}, Calc. Var. and PDE. Phys. \textbf{54}  (2015), 4165–4196.

\bibitem{LS:15} R. S. Laugesen, B. A. Siudeja, \emph{Magnetic spectral bounds on starlike plane domains}, ESIAM Control Optim. Calc. Var. \textbf{21} (2015), 670-689.


\bibitem{Lauret:19} E. Lauret, \emph{The smallest Laplace eigenvalue of homogeneous 3-spheres}, Bull. Lond. Math. Soc. \textbf{51} (2019), 49--69.


\bibitem{Li80} P. Li, \emph{On the {S}obolev constant and the {$p$}-spectrum of a compact {R}iemannian manifold}, Ann. Sci. \'{E}cole Norm. Sup. {\bf{13}} (1980), 451-468.



\bibitem{Paq79} L.~Paquet, \emph{M\'{e}thode de s\'{e}paration des variables et calcul de spectre d'op\'{e}rateurs sur les formes diff\'{e}rentielles},
   C. R. Acad. Sci. Paris S\'{e}r. {\bf{289}} (1979), A107--A110.

\bibitem{Pet98} P.~Petersen, \emph{Riemannian geometry}, Graduate Texts in Mathematics {\bf{171}}, Springer-Verlag, New York, 1998.

\bibitem{norbert-thesis} N.~Peyerimhoff,
	\emph{Ein Indexsatz f\"ur Cheegersingularit\"atten in Himblick auf algebraisce Fl\"achen},
	PhD Thesis, Augsburg 1993.

\bibitem{RS:11} S. Raulot, A. Savo, \emph{A Reilly formula and eigenvalue estimates for differential forms}, J. Geom. Anal. {\bf 21} (2011), 620-640.	
	
\bibitem{S:09} A. Savo, \emph{On the lowest eigenvalue of the Hodge Laplacian
on compact, negatively curved domains}, Ann. Glob. Anal. Geom. {\bf 35} (2009), 39-62.	

\bibitem{Sc:95} G. Schwarz, \emph{Hodge Decomposition—A Method for Solving Boundary Value Problems}, Springer, Berlin, 1995.

\bibitem{Sh87} I. Shigekawa, \emph{Eigenvalue problems for the Schr\"odinger operator with the magnetic field on a compact Riemannian manifold}, J. Funct. Anal. {\bf 75} (1987), no.
1, 92-127.
	
\bibitem{tanno:79} S.~Tanno,
	\emph{The first eigenvalue of the Laplacian on spheres},
	Tohoku Math. Journ. \textbf{31} (1979), 179--185.

\bibitem{Wu17} H.-H.~Wu, \emph{The {B}ochner technique in differential geometry}, CTM. Classical Topics in Mathematics {\bf{6}}, Higher Education Press, Beijing, 2017.

\end{thebibliography}
\end{document}